\documentclass[final]{siamltex}

\pdfoutput=1

\usepackage{epsfig}
\usepackage[english]{babel}
\usepackage{amssymb,amsmath,amsfonts}
\usepackage{mathrsfs}
\usepackage{mdwlist}
\usepackage{graphicx,enumerate,url,hyperref,enumitem,multirow,multicol}
\usepackage[usenames,dvipsnames,svgnames]{xcolor}
\hypersetup{pdfauthor={Stephen Becker, Jalal Fadili, Peter Ochs},
    pdftitle={On Quasi-Newton Forward-Backward Splitting},
    colorlinks=true,
    linkcolor=blue,
    citecolor=red,
    urlcolor=blue,
}

\usepackage{algorithm}
\usepackage{algcompatible}
\algnewcommand\algorithmicreturn{\textbf{return}}
\algnewcommand\RETURN{\State \algorithmicreturn}%
\usepackage[outercaption]{sidecap}

\usepackage{tikz}
\usepackage{bm}

\usepackage{pgfplots}
\pgfplotsset{
  tick label style={font=\footnotesize},
  label style={font=\footnotesize},
  legend style={font=\footnotesize}
}

\usepackage{booktabs}
\usepackage{caption,subcaption}

\usepackage{todonotes}

\newtheorem{example}[theorem]{Example}
\newtheorem{remark}[theorem]{Remark}

\newcommand{\BBa}{\tau_{\text{BB}2}}
\newcommand{\BBb}{\tau_{\text{BB}1}}

\newcommand{\<}{\langle}  
\renewcommand{\>}{\rangle}

\newcommand{\vx}{{x}}

\newcommand{\RR}{\mathbb{R}}
\newcommand{\R}{\mathbb{R}}

\newcommand{\NN}{\mathbb{N}}

\newcommand{\f}{h}
\newcommand{\lfun}{\mathcal{L}}
\newcommand{\jacl}{J\lfun}
\newcommand{\genjacl}{\mathcal{G}}

\newcommand{\be}{\begin{eqnarray}}
\newcommand{\ee}{\end{eqnarray}}

\renewcommand{\P}{\mathsf{P}}

\newcommand{\C}{\mathcal{C}}
\renewcommand{\int}{\mathrm{int}}

\newcommand{\conv}{\mathrm{conv}}

\newcommand{\norm}[1]{\left\|#1\right\|}
\newcommand{\anorm}[1]{\big\|#1\big\|}
\newcommand{\pds}[2]{\left\<#1,#2\right\>} \newcommand{\scal}[2]{\pds{#1}{#2}}
\newcommand{\abs}[1]{\left|#1\right|}
\newcommand{\opnorm}[1]{\big|\!\big|\!\big|#1\big|\!\big|\!\big|}
\newcommand{\veps}{{\varepsilon}}

\newcommand{\parenth}[1]{\left({#1}\right)}
\newcommand{\bpa}[1]{\big({#1}\big)}

\newcommand{\acc}[1]{\left\{{#1}\right\}}
\newcommand{\setcond}[2]{\acc{#1:~ #2}}

\newcommand{\Hm}{\mathcal{H}}

\newcommand{\beqn}{\begin{eqnarray}}
\newcommand{\eeqn}{\end{eqnarray}}

\renewcommand{\ge}{\geqslant}
\renewcommand{\leq}{\leqslant}

\newcommand{\Id}{\mathrm{Id}}

\renewcommand{\iff}{\Leftrightarrow}
\newcommand{\eqdef}{:=}
\newcommand{\sdp}{\mathbb{S}_{++}}

\newcommand{\eg}{\textit{e.g.}~}
\newcommand{\ie}{\textit{i.e.}~}

\usepackage{xspace}
\makeatletter
\newcommand*{\etc}{%
    \@ifnextchar{.}%
    {etc}%
    {etc.\@\xspace}%
}
\makeatother

\newcommand{\xs}{x^\star}
\newcommand{\xk}{x_k}
\newcommand{\xkk}{x_{k+1}}
\newcommand{\step}{\kappa}
\newcommand{\stk}{\step_k}
\newcommand{\dk}{D_k}
\newcommand{\ek}{E_k}
\newcommand{\ekk}{E_{k+1}}
\newcommand{\delk}{\Delta_k}
\newcommand{\cnd}{c}

\DeclareMathOperator{\prox}{prox}

\DeclareMathOperator{\proj}{proj}
\newcommand{\infc}{\stackrel{\mathrm{+}}{\vee}}
\newcommand{\indic}{\iota}
\newcommand{\env}[2]{{}^{#2\!}{#1}}
\DeclareMathOperator{\dom}{dom}
\DeclareMathOperator{\Span}{Im}
\DeclareMathOperator{\Ker}{Ker}

\DeclareMathOperator{\ri}{ri}

\DeclareMathOperator*{\argmin}{argmin}
\DeclareMathOperator*{\Argmin}{Argmin}

\DeclareMathAlphabet{\mathbbb}{U}{bbold}{m}{n}
\newcommand{\ones}{{\mathbf 1}}                

\usepackage[normalem]{ulem} 

\newcommand{\JF}[1]{{{#1}}}

\newcommand{\SB}[1]{{{#1}}}

\newcommand{\PO}[1]{{{#1}}}
\newcommand{\POc}[2]{\colorlet{thisColorPeter}{.}#1#2\color{thisColorPeter}}

\newif\ifcompilePGFfigs
\compilePGFfigstrue

\graphicspath{{figs/}}

\title{On Quasi-Newton Forward--Backward Splitting: \\ Proximal Calculus and Convergence}

\author{Stephen Becker\thanks{Applied Mathematics, University of Colorado Boulder (\tt{stephen.becker@colorado.edu}).}
\and
Jalal Fadili\thanks{Normandie Univ, ENSICAEN, CNRS, GREYC, France (\tt{Jalal.Fadili@greyc.ensicaen.fr}).} 
\and
Peter Ochs\thanks{Saarland University, Saarbr\"ucken, Germany (\tt{ochs@math.uni-sb.de}).}
}
\date{\today}

\begin{document}

\maketitle

\begin{abstract}
We introduce a framework for quasi-Newton forward--backward splitting algorithms (proximal quasi-Newton methods) with a metric induced by diagonal $\pm$ rank-$r$ symmetric positive definite matrices. 
This special type of metric allows for a highly efficient evaluation of the proximal mapping. The key to this efficiency is a general proximal calculus in the new metric. By using duality, formulas are derived that relate the proximal mapping in a rank-$r$ modified metric to the original metric. We also describe efficient implementations of the proximity calculation for a large class of functions; the implementations exploit the piece-wise linear nature of the dual problem. Then, we apply these results to acceleration of composite convex minimization problems, which leads to elegant quasi-Newton methods for which we prove convergence. 
The algorithm is tested on several numerical examples and compared to a comprehensive list of alternatives in the literature. Our quasi-Newton splitting algorithm with the prescribed metric compares favorably against state-of-the-art. The algorithm has extensive applications including signal processing, sparse recovery, machine learning and classification to name a few.
\end{abstract}

\begin{keywords}
forward-backward splitting, quasi-Newton, proximal calculus, duality.
\end{keywords}

\begin{AMS}
65K05, 65K10, 90C25, 90C31.
\end{AMS}


\section{Introduction}
\label{sec:intro}
Convex optimization has proved to be extremely useful to all quantitative disciplines  of science. A common trend in modern science is the increase in size of datasets, which drives the need for more efficient optimization schemes. For large-scale unconstrained smooth convex problems, two classes of methods have seen the most success: limited memory quasi-Newton methods and non-linear conjugate gradient (CG) methods. Both of these methods generally outperform simpler methods, such as gradient descent. However, many problems in applications have constraints or should be modeled naturally as non-smooth optimization problems. 

A problem structure that is sufficiently broad to cover many applications in machine learning, signal processing, image processing, computer vision (and many others) is the minimization of the sum of two convex function, one being smooth and the other being non-smooth and ``simple'' in a certain way. The gradient descent method has a natural extension to these structured non-smooth optimization problems, which is known as \emph{proximal gradient descent} (which includes projected gradient descent as a sub-case) or \emph{forward--backward splitting} \cite{BauschkeCombettes11}. Algorithmically, besides a gradient step with respect to the smooth term of the objective, the generalization requires to solve proximal subproblems with respect to the non-smooth term of the objective. The property ``simple'' from above refers the proximal subproblems. In many situations, these subproblems can be solved analytically or very efficiently. However, a change of the metric, which is the key feature of quasi-Newton methods or non-linear CG, often leads to computationally hard subproblems.

While the convergence of proximal quasi-Newton methods has been analyzed to some extent in the context of variable metric proximal gradient methods, little attention is paid to the efficient evaluation of the subproblems in the new metric. In this paper, we emphasize the fact that quasi-Newton methods construct a metric with a special structure: the metric is successively updated using low rank matrices. We develop efficient calculus rules for a general rank-$r$ modified metric. This allows \JF{popular} quasi-Newton methods, such as the SR1 (symmetric rank-1) and the L-BFGS methods, to be efficiently applied to structured non-smooth problems. The SR1 method pursues a rank-1 update of the metric and the L-BFGS method uses a rank-2 update. 

We consider the results in this paper as a large step toward the applicability of quasi-Newton methods with a comparable efficiency for smooth and structured non-smooth optimization problems.

\subsection{Problem statement}
\label{sec:statement}
Let $\Hm=(\RR^N,\pds{\cdot}{\cdot})$ equipped with the usual Euclidean scalar product $\pds{x}{y}=\sum_{i=1}^N x_iy_i$ and associated norm $\norm{x}=\sqrt{\pds{x}{x}}$. For a matrix $V \in \RR^{N \times N}$ in the symmetric positive-definite (\SB{SPD}) cone $\sdp(N)$, we define $\Hm_V=(\RR^N,\pds{\cdot}{\cdot}_V)$ with the scalar product $\pds{x}{y}_V = \pds{x}{Vy}$ and norm $\norm{x}_V$ corresponding to the metric induced by $V$. The dual space of $\Hm_V$, under $\pds{\cdot}{\cdot}$, is $\Hm_{V^{-1}}$. We denote \SB{the identity operator as $\Id$}. \JF{For a matrix $A$, $A^+$ is its Moore-Penrose pseudo-inverse. For a positive semi-definite matrix $A$, $A^{1/2}$ denotes its principal square root.}

An extended-valued function $f: \Hm \to \RR \cup \acc{+\infty}$ is (0)-\emph{coercive} if $\lim_{\norm{\vx} \to +\infty}f\parenth{\vx}=+\infty$. The \emph{domain} of $f$ is defined by $\dom f = \{ x\in\Hm\ :\ f(x) < +\infty \}$ and $f$ is \emph{proper} if $\dom f \neq
\emptyset$. We say that a real-valued function $f$ is \emph{lower semi-continuous} (lsc) if $\liminf_{x \to x_0} f(x) \geq f(x_0)$. The class of all proper lsc convex functions from $\Hm$ to $\RR \cup \acc{+\infty}$ is denoted by $\Gamma_0(\Hm)$. The conjugate or Legendre-Fenchel transform of $f$ on $\Hm$ is denoted $f^*$.

Our goal is the generic minimization of functions of the form 
\begin{equation}\tag{$\P$}
\label{eq:minP}
\min_{x \in \Hm} ~ \{F(x) \eqdef f(x) + h(x)\}~,
\end{equation}
where $f,h \in \Gamma_0(\Hm)$. 
We also assume the set of minimizers $\Argmin(F)$ is nonempty. Write $\xs$ to denote an element of $\Argmin(F)$. We assume that 
$f \in C^{1,1}(\SB{\Hm})$,
 meaning that it is continuously differentiable and its gradient (in $\Hm$) is $L$-Lipschitz continuous.

The class we consider covers structured smooth+non-smooth convex optimization problems, including those with convex constraints.
Here are some examples in regression, machine learning and classification.
\begin{example}[LASSO] \label{ex:intro-example-lasso} 
  Let $A$ be a matrix, \SB{$\lambda>0$, and $b$ a vector} of appropriate dimensions.
    \begin{equation}
        \min_{x \in \Hm} \frac{1}{2}\|Ax-b\|_2^2 + \lambda \|x\|_1 ~.
        \label{eq:LASSO}
    \end{equation}
\end{example}
\begin{example}[Non-negative least-squares (NNLS)]
  Let $A$ and $b$ be as in Example~\ref{ex:intro-example-lasso}.
        \begin{equation}
            \min_{x \in \Hm} \frac{1}{2}\|Ax-b\|_2^2 \quad\text{subject to}\quad x \ge 0 ~.
        \label{eq:NNLS}
    \end{equation}
\end{example}
\begin{example}[Sparse Support Vector Machines]
One would like to find a linear decision function which minimizes the objective
    \begin{equation}
        \min_{x \in \Hm, b \in \RR} \frac{1}{m}\sum_{i=1}^m \mathscr L(\pds{x}{z_i}+b,y_i) + \lambda \|x\|_1
        \label{eq:hingesvm}
    \end{equation}
where for 
$i=1,\cdots,m$, $(z_i,y_i) \in \SB{\Hm} \times \{\pm 1\}$
is the training set, and $\mathscr L$ is a smooth loss function with Lipschitz-continuous gradient such as the squared hinge loss $\mathscr L(\hat{y}_i,y_i)=\max(0,1-\hat{y}_iy_i)^2$ or the logistic loss $\mathscr L(\hat{y}_i,y_i)=\log(1+e^{-\hat{y}_iy_i})$. The term $\lambda\|x\|_1$ promotes sparsity of the decisive features steered by a parameter $\lambda>0$.
\end{example}

\subsection{Contributions} 

We introduce an general proximal calculus in a metric $V=P\pm Q\in \sdp(N)$ given by $P\in\sdp(N)$ and a positive semi-definite rank-$r$ matrix $Q$. This significantly extends the result in the preliminary version of this paper \cite{BF12}, where only $V=P+Q$ with a rank-$1$ matrix $Q$ is addressed. The general calculus is accompanied by several more concrete examples (see Section~\ref{sec:examples} for a non-exhaustive list), where, for example, the piecewise linear nature of certain dual problems is rigorously exploited. 

Motivated by the discrepancy between constrained and unconstrained performance, we define a class of limited-memory quasi-Newton methods to solve \eqref{eq:minP} \SB{which} extends naturally and elegantly from the unconstrained to the constrained case. In particular, we generalize the zero-memory SR1 and L-BFGS quasi-Newton methods to the proximal quasi-Newton setting for solving \eqref{eq:minP}, and prove their convergence. Where L-BFGS-B~\cite{LBFGSB} is only applicable to box constraints, our quasi-Newton methods efficiently apply to a wide-variety of non-smooth functions.

\SB{To clarify the differences between this paper and the conference paper \cite{BF12}, the current paper (1) extends the proximal framework to allow $V=P\pm Q$ scalings where $Q$ is rank $r\ge 1$ (Theorem~\ref{theo:proxVrankr}, and specialized to the $r=1$ case in Theorem~\ref{theo:proxVrank1}), using Toland duality to handle non-convexity issues that arise in the $P-Q$ case, whereas \cite{BF12} considers only $V=P+Q$ for $Q$ rank-1 and positive semi-definite; (2) discusses at length bisection and semi-smooth methods to solve the dual problem, and gives global (Proposition~\ref{prop:bisectionConverges}) and local (Proposition~\ref{prop:generaldh-rank-r}) convergence results, respectively; (3) introduces the zero-memory L-BFGS quasi-Newton forward-backward algorithm (Algorithm~\ref{alg:LBFGS}) in addition to the SR1 one; (4) proves convergence results for these algorithms (Theorems~\ref{thm:convergence} and \ref{thm:convergence-bfgs}, respectively); and (5) discusses a few new examples of non-separable proximity operator including that of the $\ell_1-\ell_2$ norm in Section~\ref{sec:examples} and runs numerical experiments with this norm in Section~\ref{sec:num-exp-l1l2-lasso}.}

\subsection{Paper organization}

Section~\ref{sec:quasi-Newton-FBS} formally introduces quasi-Newton methods and their generalization to the structured non-smooth setting \eqref{eq:minP}. The related literature is extensively discussed. In order to obtain a clear perspective on how to apply the proximal calculus that is developed in Section~\ref{sec:prox}, the outline of our proposed zero-memory SR1 and our zero-memory BFGS quasi-Newton method is provided in Section~\ref{sec:quasi-Newton-FBS}. The main result that simplifies the rank-$r$ modfied proximal mapping is stated in Section~\ref{sec:rank-r-mod-prox}, followed by several specializations and an efficient semi-smooth Newton-based root finding strategy that is required in some situations. Section~\ref{sec:SR1} describes the details for the construction of the SR1 metric and states the convergence result. Following the same outline, the L-BFGS metric is constructed in Section~\ref{sec:BFGS} and convergence is proved. The significance of our results is confirmed in numerical experiments.


\section{Quasi-Newton forward--backward splitting} \label{sec:quasi-Newton-FBS}

\subsection{The algorithm}

The main update step of our proposed algorithm for solving \eqref{eq:minP} is a forward--backward splitting (FBS) step in a special type of metric. In this section, we introduce the main algorithmic step and Section~\ref{sec:prox} shows that our choice of metric allows the update to be computed efficiently.

We define the following quadratic approximation to the smooth part $f$ of the objective function in \eqref{eq:minP} around the current iterate $x_k$
\begin{equation}\label{eq:Q}
    Q_\step^B(x; x_k) := f(x_k) + \pds{\nabla f(x_k)}{ x-x_k} + \frac{1}{2\step}\|x-x_k\|_{B}^2\,, 
\end{equation}
where $B \in \sdp(N)$ and $\step>0$. The (non-relaxed) version of the variable metric FBS algorithm (also known as proximal gradient descent) to solve \eqref{eq:minP} updates to a new iterate $x_{k+1}$ according to
\begin{equation} \label{eq:firstOrder}
    x_{k+1} = \argmin_{x \in \RR^N} Q_{\stk}^{B_k}(x;x_k) + h(x) =: \prox^{B_k}_{\stk h}( x_k - \stk B_k^{-1}\nabla f(x_k) ) 
\end{equation}
with (iteration dependent) step size $\stk$ and metric $B_k\in \sdp(N)$. The right hand side uses the so-called proximal mapping, which is formally introduced in Definition~\ref{def:prox}. Standard results (see, e.g., \cite{CV14,vu2013variable}) show that, for a sequence $(B_k)_{k \in \NN}$ that varies moderately (in the Loewner partial ordering sense) such that $\inf_{k \in \NN} \norm{B_k}=1$, convergence of \JF{the sequence} $(x_k)_{k \in \NN}$ is expected when $0 < \underline{\step} \leq \stk \leq \overline{\step} < 2/L$, \SB{where $L$ is the Lipschitz constant of $\nabla f$.}

Note that when $h=0$, \eqref{eq:firstOrder} reduces to gradient descent if $B_k=\Id$, which is a poor approximation and requires many iterations, but each step is cheap.
When $f$ is also $C^2(\RR^N)$, the Newton's choice $B_k=\nabla^2 f(x_k)$ is a more accurate approximation and reduces to Newton's method when $h=0$. The update step is well-defined \JF{(at least locally) if $\nabla^2 f(x^\star)$ is positive-definite}, but may be computationally demanding as it requires solving a linear system and possibly storing the Hessian matrix. Yet, because it is a more accurate approximation, Newton's method has local quadratic convergence under standard assumptions such as self-concordancy.
Motivated by the superiority of Newton and quasi-Newton methods over gradient descent for the case $h=0$, we pursue a quasi-Newton approximation for $B_k$ \JF{for the case $h \neq 0$}. However, the update is now much more involved than just solving a linear system. Indeed, one has to compute the proximal mapping in the metric $B_k$, which is, in general, as difficult as solving the original problem~\eqref{eq:minP}. For this reason, we restrict $B_k$ to \JF{the structured form of a positive-definite "simple" matrix (e.g., diagonal) plus or minus a low-rank term.}

The main steps of our general quasi-Newton forward--backward scheme to solve \eqref{eq:minP} are given in Algorithm~\ref{alg:framework}. Its instantiation for a diagonal $-$ rank 1 metric (0SR1) and a diagonal $-$ rank 2 metric (0BFGS) are respectively listed in Algorithm~\ref{alg:main} and Algorithm~\ref{alg:LBFGS}. Details for the selection of the corresponding metrics are provided in Section~\ref{sec:SR1} and~\ref{sec:BFGS}. Following the convention in the literature on quasi-Newton methods, throughout the paper, we use $B_k$ as an approximation to the Hessian and $H_k:=B_k^{-1}$ as \SB{the approximation to its inverse}. The algorithms are listed as simply as possible to emphasize the important components; the actual software used for numerical tests is open-source and available at {\small\url{https://github.com/stephenbeckr/zeroSR1}}.

\POc{}{In Sections~\ref{sec:SR1} and~\ref{sec:BFGS}, we will prove Algorithm~\ref{alg:framework} converges linearly under the assumption that $f$ is strongly convex and $t=1$, which is the standard theoretically controllable setting for Newton and quasi-Newton methods. Moreover, global convergence of subsequences to a minimizer for the line-search variant can be deduced from the literature \cite{Salzo16,BLPP16,OFB17}. Thanks to the line search, the choice of the metric \emph{need not} obey monotonicity. If standard assumptions on the monotonicity of the metric are satisfied, convergence to a minimizer can be proved \cite{Salzo16,BLPP16}. Moreover, the convergence results in \cite{BLPP16} account for inexact evaluation of the proximal mapping, which even allows us to invoke a semi-smooth Newton Method for solving the subproblems numerically (see Section~\ref{subsec:semi-smooth-Newton}).}

\begin{algorithm}[h]
    \caption{Quasi-Newton forward--backward framework to solve \eqref{eq:minP}
         \label{alg:framework} }
\begin{algorithmic}[1]
\REQUIRE $x_0\in\dom(f+h)$, 
 stopping criterion $\epsilon$,
method to compute stepsizes $t$ and $\stk$ (\eg based on the Lipschitz constant estimate $L$ of $\nabla f$ and strong convexity $\mu$ of $f$)
  \FOR{$k=1,2,3,\dots$}
  \STATE $s_k \leftarrow x_k - x_{k-1}$
  \STATE $y_k \leftarrow \nabla f(x_k) - \nabla f(x_{k-1})$
  \STATE Compute $H_k$ according to a quasi-Newton framework
  \STATE Define $B_k = H_k^{-1}$ and compute the variable metric proximity operator (see Section~\ref{sec:prox}) with stepsize $\stk$
  \begin{equation}
  \label{eq:fbsr1}
      \bar{x}_{k+1} \leftarrow  \prox^{B_k}_{\stk h}( x_k - \stk H_k \nabla f(x_k) )  
  \end{equation}
  \STATE $p_k \leftarrow \bar{x}_{k+1} - x_k$ and 
  terminate if $\|p_k\| < \epsilon $
  \STATE Line-search along the ray $x_k + t p_k$ to determine $x_{k+1}$, or choose $t=1$.
  \ENDFOR
\end{algorithmic}  
\end{algorithm}

\begin{algorithm}[h]
    \caption{Zero-memory Symmetric Rank 1 (0SR1) algorithm to solve \eqref{eq:minP}, cf.\ Section~\ref{sec:SR1} \label{alg:main} }
    \begin{algorithmic}[1]
        \REQUIRE as for Algorithm~\ref{alg:framework}, and parameters $\gamma,\tau_\text{min},\tau_\text{max}$ for Algorithm~\ref{alg:SR1}
        \STATEx Iterate as in Algorithm~\ref{alg:framework}, with line 4 as:
\makeatletter
\setcounter{ALG@line}{3}
\makeatother
        \STATE Compute $H_k$ via Algorithm~\ref{alg:SR1} (diagonal plus rank one)
    \end{algorithmic}  
\end{algorithm}

\begin{algorithm}[h]
    \caption{Zero-memory BFGS (0BFGS) algorithm to solve \eqref{eq:minP}, cf.\ Section~\ref{sec:BFGS} \label{alg:LBFGS} }
    \begin{algorithmic}[1]
        \REQUIRE as for Algorithm~\ref{alg:framework} 
        \STATEx Iterate as in Algorithm~\ref{alg:framework}, with line 4 as:
        \makeatletter
        \setcounter{ALG@line}{3}
        \makeatother
        \STATE Compute $H_k$ via Eq.~\eqref{eq:BFGS-inv-Hess-formula-derivation}) (diagonal plus rank two)
    \end{algorithmic}  
\end{algorithm}

\begin{remark} \label{rem:pm-rank-1-confusion}
    The usage of the terms ``diagonal $-$ rank $r$'' and ``diagonal $+$ rank $r$'' \SB{needs} clarification. The meaning of these terms is that $B_k=D-\sum_{i=1}^r u_i u_i^\top$ or $B_k=D+\sum_{i=1}^r u_i u_i^\top$, respectively, where $D$ is a diagonal matrix and $u_i \in \R^N$. Collectively, to cover both cases, $B_k=D \pm \sum_{i=1}^r u_i u_i^\top$ is used. Algorithmically, the choice of ``$+$'' or ``$-$'' is crucial.

  For instance, if we talk about a ``diagonal $\pm$ rank 1 quasi-Newton method'', this taxonomy applies to the approximation of the Hessian $B_k$. Since, the inverse $H_k$ can be computed conveniently with the Sherman--Morrison inversion lemma, it is also of type ``diagonal $\mp$ rank 1'', where the sign of the rank 1 part is flipped. 
  The analysis in \cite{BF12} of the rank 1 proximity operator applied to the case ``diagonal $+$ rank 1''. In this paper, we cover both cases ``diagonal $\pm$ rank 1'', which generalizes and formalizes the ``diagonal $-$ rank 1'' setting in \cite{KV17}.
\end{remark}

\subsection{Relation to prior work}
\label{sec:relation-prior-work}

\paragraph{First-order methods}
The algorithm in \eqref{eq:firstOrder} with $B_k=\Id$ is variously known as proximal \SB{gradient} descent or iterated shrinkage/thresholding algorithm (IST or ISTA). It has a grounded convergence theory, and also admits over-relaxation factors $\alpha \in (0,1)$~\cite{CombettesPesquetChapter}.

The spectral projected gradient (SPG)~\cite{SPG} method was designed as an extension of the Barzilai--Borwein spectral step-length method to constrained problems. 
 In~\cite{WrightSparsa08}, it was extended to non-smooth problems by allowing general proximity operators.  
The Barzilai--Borwein method~\cite{BB88} uses a specific choice of step-length $\stk$ motivated by quasi-Newton methods. Numerical evidence suggests the SPG/Sp\-aRSA method is highly effective, although convergence results are not as strong as for ISTA. 

FISTA~\cite{FISTA} is a (two-step) inertial version of ISTA inspired by the work of Nesterov \cite{Nest83}. It can be seen as an explicit-implicit discretization of a nonlinear second-order dynamical system (oscillator) with viscous damping that vanishes asymptotically in a moderate way~\cite{su2014differential,AttouchFISTA16}. While the stepsize $\step$ is chosen in a similar way to ISTA (though with a smaller upper-bound), in our implementation, we tweak the original approach by using a Barzilai--Borwein step size, a standard line search, and restart \cite{DC15}, since this led to improved performance. 

Recently, \cite{OP17} has shown that optimizing the inertial parameter in each iteration of FISTA, applied to the sum of a quadratic function and a non-smooth function, the method is equivalent to the zero memory SR1 proximal quasi-Newton method that we propose in Section~\ref{sec:SR1}. Convergence is analyzed with respect to standard step sizes that relate to the Lipschitz constant, which does not cover the case of Barzilai--Borwein step size.

The above approaches assume $B_k$ is a constant diagonal. The general diagonal case was considered in several papers in the 1980s as a simple quasi-Newton method, but never widely adapted. Variable metric operator splitting methods have been designed to solve monotone inclusion problems and convex minimization problems, see for instance~\cite{CV14,vu2013variable} in the maximal monotone case and~\cite{chen1997convergence} for the strongly monotone case. The convergence proofs rely on a variable metric extension of quasi-Fej\'{e}r monotonicity~\cite{CV13}. In particular, this requires the variable metric to be designed a priori to verify appropriate growth conditions. However, it is not clear how to make the metric adapt to the geometry of the problem. In fact, in practice, the metric is usually chosen to be diagonal for the proximity operator to be easily computable. When the metric is not diagonal but fixed, these methods can be viewed as pre-conditioned versions that were shown to perform well in practice for certain problems (i.e. functions $h$)~\cite{ChambollePock11b,bredies2015preconditioned}. But again, the choice of the metric (pre-conditioner) is quite limited for computational and storage reasons.

\paragraph{Active set approaches} Active set methods take a simple step, such as gradient projection, to identify active variables, and then uses a more advanced quadratic model to solve for the free variables. A well-known such method is L-BFGS-B \cite{LBFGSB,L-BFGS-B-97} which handles general box-constrained problems; we test an updated version~\cite{LBFGSB2011}. 
A recent bound-constrained solver is ASA~\cite{HagerZhang06} which uses a conjugate gradient (CG) solver on the free variables, and shows good results compared to L-BFGS-B, SPG, GENCAN and TRON.
We also compare to several active set approaches specialized for $\ell_1$ penalties:
``Orthant-wise Learning'' (OWL)~\cite{AndrewGao07}, 
``Projected Scaled Sub-gradient + Active Set'' (PSSas)~\cite{Schmidt2007a},
``Fixed-point continuation + Active Set'' (FPC\_AS)~\cite{FPCAS},
and ``CG + IST'' (CGIST)~\cite{GoldsteinSetzer11}.

\paragraph{Other approaches}
By transforming the problem into a standard conic programming problem, the generic problem is amenable to interior-point methods (IPM). IPM requires solving a Newton-step equation, so first-order like ``Hessian-free'' variants of IPM solve the Newton-step approximately, either by approximately solving the equation or by subsampling the Hessian. The main issues are speed and robust stopping criteria for the approximations.  

Yet another approach is to include the non-smooth $h$ term in the quadratic approximation. Yu et al.~\cite{YuVishwanathan10} propose a non-smooth modification of BFGS and L-BFGS, and test on problems where $h$ is typically a hinge-loss or related function. Although convergence of this method cannot be expected in general, there are special cases for which convergence results could be established \cite{LO13,LZ15}, and more recently \cite{GL17}. The empirically justified good numerical performance
has been observed for 
decades \cite{Lemarechal82}.

The projected quasi-Newton (PQN) algorithm~\cite{projQuasiNewton09,projNewton11} is perhaps the most elegant and logical extension of quasi-Newton methods, but it involves solving a sub-iteration or need to be restricted to a diagonal metric in the implementation \cite{BZZ09,BP15}. PQN proposes the SPG~\cite{SPG} algorithm for the subproblems, and finds that this is an efficient trade-off whenever the cost function (which is not involved in the sub-iteration) is significantly more expensive to evaluate than projecting onto the constraints. Again, the cost of the sub-problem solver (and a suitable stopping criteria for this inner solve) are issues. 
The paper \cite{friedlander2017efficient} shows how the sub-problem can be solved efficiently by a special interior-point method when $h$ is a quadratic-support function. 
As discussed in~\cite{Saunders14}, it is possible to generalize PQN to general non-smooth problems whenever the proximity operator is known (since, as mentioned above, it is possible to extend SPG to this case). In the same line of methods, \cite{BLPP16} proposes a flexible proximal quasi-Newton method that extends \cite{BP15} to simple proximal operators, though a diagonal metric is considered in the implementation. Another work that unifies and generalizes several of the works mentioned above in a variable metric (i.e. quasi-Newton) setting is \cite{Salzo16}.

A more general and efficient step size strategy with memory was proposed in \cite{Fletcher11} for unconstrained optimization, which was generalized to a scaled gradient projection method in \cite{PPZ15}, and used in the proximal gradient method in \cite{BLPPR16}. However, the flexible choice of the step size and the scaling of the metric is not for free when convergence guarantees are sought. \cite{BLPPR16,BLPP16} rely on a line search strategy to account for a descent of the objective values. The metric in \cite{CPR13} is constructed such that \eqref{eq:Q} is a majorizer of the (possibly non-convex) objective and the step size selection is more conservative, however line search can be avoided. 

The works \cite{PSB14,STP17} make use of the so-called forward--backward envelope, a concept that allows them to reinterpret the forward--backward splitting algorithm as a variable metric gradient method for a smooth optimization problem. Using this reformulation, they can apply classical Newton or quasi-Newton methods. Proximal quasi-Newton methods have also been considered in combination with the Heavy-ball method \cite{Ochs16}, and have been generalized further. 

The proximal quasi-Newton methods described so far simply assume that the scaled proximal mapping can be solved efficiently, rely on solving subproblems, or simple diagonal scaling metrics. The first work on systematically solving non-diagonally scaled proximal mappings efficiently is the conference version of this paper \cite{BF12}. The key is structure of the metric. In \cite{BF12}, it is assumed to be given as the sum of a diagonal and a rank-1 matrix. For the special case of the $\ell_1$-norm, the approach was transferred to the difference of a diagonal and a rank-1 matrix in \cite{KV17}. A systematic analysis for both cases where a rank-$r$ modification is allowed, is presented in this paper. 

\PO{The key result for efficiently computing the proximal mapping in this paper reveals a decomposition into a simple proximal mapping (for example, w.r.t. a diagonal metric) and a low-dimensional operator equation (root finding problem). In several cases, the operator equation can be solved exactly using specialized techniques. In the general case, we rely on a semi-smooth Newton strategy. It is known that the convergence of the latter, under mild conditions, is remarkably (locally) super-linear~\cite{FacchineiPang03}, which may even be improved to quadratic convergence under strong semi-smoothness \cite{QS93}. A similar result was independently obtained in \cite{Kum92} under similar assumptions. 

Due to the great success of Newton's method for smooth equations, the non-smooth setting has been actively studied and is still the subject of ongoing research, see for instance the recent monograph \cite{Ulbrich11}. Early studies of generalizing Newton's method for solving non-smooth equations include \cite{KS86} for piecewise smooth equations, \cite{Pang90,Rob94} for so-called B-differentiable equations and \cite{Kum88} for locally Lipschitz functions. As pointed out in \cite{QS93}, semi-smoothness is a crucial property in the super-linear convergence analysis of these methods. Semi-smooth Newton methods have also been adapted to non-smooth operator equations in function spaces \cite{Ulbrich02}. Recognizing semi-smoothness is however not always immediate. In~\cite{BolteSN09}, the authors proposed a large class of semi-smooth mappings. Our convergence results on the semi-smooth Newton method will then rely on \cite{FacchineiPang03,BolteSN09}.} 




\section{Proximal calculus in \texorpdfstring{$\Hm_V$}{Hv}}
\label{sec:prox}

A key step for efficiently implementing Algorithm~\ref{alg:framework} is the evaluation of the proximity operator in \eqref{eq:fbsr1}. Even if the proximal mapping $\prox_{\f}$ can be computed efficiently, in general, this is not true for $\prox_{\f}^{V}$. However, we construct $V$ of the form ``diagonal $\pm$ rank $r$'', for which we propose an efficient calculus in this section. In order to cover this topic broadly, we assume $V=P\pm Q$ is a rank-$r$ modification $Q$ of a matrix $P$. The main result (Theorem~\ref{theo:proxVrankr}) shows that the proximity operator $\prox_{\f}^V$ in the modified metric $V$ can be reduced essentially to the proximity operator $\prox_{\f}^P$ without the rank-$r$ modification and an $r$-dimensional root finding problem.

\subsection{Preliminaries}

We only recall here essential definitions. More notions, results from convex analysis as well as proofs are deferred to the appendix.

\begin{definition}[Proximity operator \cite{Moreau1962}]
\label{def:prox} 
Let $\f \in \Gamma_0(\Hm)$. Then, for every $x\in\Hm$, the
function $z \mapsto \frac{1}{2}\norm{x-z}^{2} + \f(z)$ achieves its infimum at a unique
point denoted by $\prox_{\f}(x)$. The single-valued operator $\prox_{\f}: \Hm \to \Hm$ thus
defined is the \textit{proximity operator} or proximal mapping of $\f$.
Equivalently, $\prox_{\f} = (\Id+\partial \f)^{-1} $ where $\partial \f$ is the subdifferential of $\f$. \JF{When $\f$ is the indicator function of a non-empty closed convex set $\C$, the corresponding proximity operator is the orthogonal projector onto $\C$, denoted $\proj_{\C}$.}
\end{definition}


Throughout, we denote by
\begin{equation} \label{eq:def:prox-in-Hv}
  \prox^V_\f(x) = \argmin_{z\in\Hm} \f(z) + \frac 12 \|x-z\|^2_V = (\Id+V^{-1} \partial \f)^{-1}(x) ~,
\end{equation}
 the proximity operator of $\f$ w.r.t.\ the norm endowing $\Hm_V$ for some $V \in \sdp(N)$. Note that since $V \in \sdp(N)$, the proximity operator $\prox^V_\f$ is well-defined. The proximity operator $\prox^V_\f$ can also be expressed in the metric of $\Hm$.
\begin{lemma} \label{lem:rel-prox-HV-to-H}
Let $\f \in \Gamma_0(\Hm)$ and $V\in \sdp(N)$. Then, the following holds:
\[
  \prox^V_\f (x) = V^{-1/2}\circ \prox_{\f\circ V^{-1/2}}\circ V^{1/2} (x) \,.
\]
\end{lemma}

The proof is in Section~\ref{appendix:lem:rel-prox-HV-to-H}.
The important Moreau identity can be translated to the space $\Hm_V$.
\begin{lemma}[Moreau identity in $\Hm_V$]
\label{lem:moreauidentV}
Let $\f \in \Gamma_0(\Hm)$, then for any $x \in \Hm$
\be
\prox^V_{\rho \f^*}(x) + \rho V^{-1} \circ \prox^{V^{-1}}_{\f/\rho} \circ V(x/\rho) = x, \forall ~ 0 < \rho < +\infty ~.
\ee
For $\rho=1$, it simplifies to 
\be
\prox^V_{\f}(x) = x - V^{-1} \circ \prox^{V^{-1}}_{\f^*} \circ V(x) ~.
\ee
\end{lemma}

The proof is in Section~\ref{appendix:lem:moreauidentV}.

\subsection{Rank-$r$ modified metric}
\label{sec:rank-r-mod-prox}
\newcommand{\matP}{P}
\newcommand{\matQ}{Q}
\newcommand{\matW}{W}
\newcommand{\smat}[1]{{\scriptsize\begin{bmatrix}#1\end{bmatrix}}}
\newcommand{\map}[3]{#1\colon #2\rightarrow#3}

\JF{In this section, we present the general result for a metric $V = \matP\pm \matQ\in \sdp(N)$, where $\matP\in\sdp(N)$ and $\matQ=\sum_{i=1}^r u_i u_i^\top \in \R^{N\times N}$ is symmetric with $\rank(\matQ)=r$ and $r \leq N$, given by $r$ linearly independent vectors $u_1,\ldots,u_r\in \Hm$. Computing the proximity operator $\prox_\f^V$ can be reduced to the simpler problem of evaluating $\prox_\f^\matP$ and an $r$ dimensional root finding problem, which can be solved either exactly (see Section~\ref{sec:prox-diag-pm-rank-1}) or by efficient fast iterative procedures with controlled complexity such as bisection (Section~\ref{subsec:rank1bisection}) or semi-smooth Newton iterations (Section~\ref{subsec:semi-smooth-Newton}).}

\subsubsection{\JF{General case}}
We start with our most general result.
\begin{theorem}[Proximity operator for a rank-$r$ modified metric]
\label{theo:proxVrankr}
Let $\f \in \Gamma_0(\Hm)$ \JF{and $V = \matP\pm \matQ\in \sdp(N)$, where $\matP\in\sdp(N)$ and $\matQ=\sum_{i=1}^r u_i u_i^\top \in \R^{N\times N}$ with $r=\rank(\matQ) \leq N$. Denote $U=(u_1,\cdots,u_r)$.} Then,
\be
\label{eq:proxVgenr}
  \begin{split}
\prox^V_\f(x) =&\  \matP^{-1/2} \circ \prox_{\f \circ \matP^{-1/2}}\circ \matP^{1/2}(x \mp \matP^{-1}U\alpha^\star) \\
              =&\ \prox^\matP_\f (x \mp \matP^{-1}U\alpha^\star) ~,
  \end{split}
\ee
\JF{where $\alpha^\star \in \RR^r$ is the unique zero of the mapping $\lfun: \RR^r \to \RR^r$}
\be
\label{eq:roothgenr}
  \begin{split}
  \lfun(\alpha) :=&\  U^\top\parenth{x - \matP^{-1/2} \circ \prox_{\f \circ \matP^{-1/2}}\circ \matP^{1/2}(x \mp \matP^{-1} U\alpha)} + \alpha \\
            =&\  U^\top\parenth{x - \prox^\matP_\f(x \mp \matP^{-1} U\alpha)} + \alpha ~.
  \end{split}
\ee
The mapping $\lfun$ is Lipschitz continuous with Lipschitz constant $1+\opnorm{\matP^{-1/2}U}^2$, \JF{and strongly monotone with modulus $c$, where $c=1$ for $V=P+Q$ and $c=1-\opnorm{\matP^{-1/2}U}^2$ for $V=P-Q$}. 
\end{theorem}

The proof is in Section~\ref{appendix:theo:proxVrankr}.

\begin{remark}
{~}
\begin{itemize}
  \item The root finding problem in Theorem~\ref{theo:proxVrankr} emerges from the dual problem for solving $\prox_\f^V$. Passing to the dual problem reduces dramatically the dimensionality of the problem to be solved from $N$ to $r$ where usually $r\ll N$.
   The dual problem boils down to an $r$-dimensional root finding problem of a strongly monotone function.
  \item Theorem~\ref{theo:proxVrankr} simplifies the computation of $\prox_\f^V$ to $\prox_\f^P$ (or equivalently $\prox_{\f\circ\matP^{-1/2}}$), which is often much easier to solve. \JF{This is typically the case when $\matP$ is a diagonal matrix as will be considered in Section~\ref{sec:prox-diag-pm-rank-1}. Another interesting scenario is when $\f = \psi \circ \matP^{1/2}$, where $\psi \in \Gamma_0(\Hm)$ is a simple function so that $\prox_{\f\circ\matP^{-1/2}} = \prox_{\psi}$ is easy to compute. Thus the matrix $\matP$ in the expression of $V$ can be interpreted as a pre-conditioner. In Section~\ref{sec:prox-diag-pm-rank-1}, we will focus on the case $\matP$ is diagonal since all standard and efficient quasi-Newton methods (e.g., SR1, L-BFGS) use a diagonal $\matP$.}
  \item The variable metric forward--backward splitting algorithm requires the inverse of the metric in the forward step. It can be computed using the Sherman-Morrison inversion lemma: If $V=\matP\pm \matQ$ with $\rank(Q)=r$, then  
  \[
      V^{-1} = P^{-1} \mp \widetilde Q^{-1}\,, \quad \widetilde Q^{-1} := P^{-1} Q\,(\Id \pm P^{-1} Q)^{-1}P^{-1}\,,
  \]
  with $\rank (\widetilde Q^{-1})=r$. Note that the sign of the rank-$r$ part flips, see also Remark~\ref{rem:pm-rank-1-confusion}.
  \item Using the inversion formula for $V=P\pm Q$ as in the preceding item, and using Lemma~\ref{lem:moreauidentV} (Moreau identity in $\Hm_V$), the computation of
   the proximity operator of the convex conjugate function $\f^*$, $\prox_{\f^*}^{V}$, can
   be cast in terms of computing $\prox_\f^{V^{-1}}$.
%
 
\end{itemize}
\end{remark}
\begin{corollary} \label{cor:proxVrankr}
  Let $V = P + Q_1 - Q_2 \in \sdp(N)$ with $P\in\sdp(N)$ 
  and symmetric positive semi-definite matrices $Q_1,Q_2$ with $\rank(Q_i)=r_i$ and let $\Span(Q_i)$ be spanned by the columns of $U_i\in \R^{N\times r_i}$, $i=1,2$. Set $P_1=P+Q_1$. Then, for $\f\in \Gamma_0(\Hm)$, the following holds:
  \[
    \prox_\f^V (x) = \prox_\f^{P_1}(x+P_{1}^{-1}U_1\alpha_1^\star) 
    = \prox_\f^{P}(x+P_{1}^{-1}U_1\alpha_1^\star - P^{-1}U_2\alpha_2^\star)
  \]
  where $\alpha_i^\star\in \R^{r_i}$, $i=1,2$, are the unique zeros of the coupled system
  \[
    \begin{split}
    \lfun_1(\alpha_1,\alpha_2) =&\  U_1^\top(x - \prox_\f^{P}(x+ P_1^{-1} U_1\alpha_1 -  P^{-1} U_2\alpha_2)) - \alpha_1 \\
    \lfun_2(\alpha_1,\alpha_2) =&\  U_2^\top(x+P_{1}^{-1}U_1\alpha_1 - \prox_\f^{P}(x+P_{1}^{-1}U_1\alpha_1- P^{-1} U_2\alpha_2)) - \alpha_2 .
    \end{split}
  \]
\end{corollary}
\begin{proof}
  Corollary~\ref{cor:proxVrankr} follows from a recursive application of Theorem~\ref{theo:proxVrankr} to
   $\prox_\f^{V}$ with $V=P_1-Q_2$ and $\prox_\f^P$ with $P_1=P+Q_1$.
\end{proof}

\JF{As discussed above, depending on the structure of the proximity operator $\prox_{\f \circ P^{-1/2}}$, either general-purpose or specialized algorithms for solving the root-finding problem can be derived. In some situations, see e.g.,~Proposition~\ref{prop:proxVseprank1}, the root of the function $\lfun$ can be found exactly in linear time. If no special structure is available, however, one can appeal to some efficient iterative method to solve \eqref{eq:roothgenr} as we see now.}

\subsubsection{Semi-smooth Newton method} \label{subsec:semi-smooth-Newton}

\JF{We here turn to the semi-smooth Newton method to solve $\lfun(\alpha)=0$ (see \eqref{eq:roothgenr}) using the fact that $\lfun$ is Lipschitz-continuous and strongly monotone (Theorem~\ref{theo:proxVrankr}).

Since $\lfun: \RR^r \to \RR^r$ is Lipschitz continuous, it is so-called Newton differentiable~\cite{Chen2000}, i.e.,
there exists a family of linear mappings $\genjacl$ (called generalized Jacobians) such that for all $\alpha$ on an open subset of $\RR^r$
\[
\lim_{d \to 0} \frac{\norm{\lfun(\alpha+d) - \lfun(\alpha) - \genjacl(\alpha+d)d}}{\norm{d}} = 0 .
\]
However, this is only of little help algorithmically unless one can construct a generalized Jacobian $\genjacl$ which is easily computable and provably invertible under our strong monotonicity assumption. This is why we turn to the semi-smoothness framework.

We shall write $\jacl(\alpha) \in \RR^{r \times r}$ for the usual Jacobian matrix whenever $\alpha$ is a point in the differentiability set $\Omega \subset \RR^r$ (its complement has measure zero by the celebrated Rademacher's theorem). The Clarke Jacobian of $\lfun$ at $\alpha \in \RR^r$ is defined as~\cite[Definition~2.6.1]{Clarke}
\[
\partial^C \lfun(\alpha) = \conv\setcond{G \in \RR^{r \times r}}{G = \lim_{\alpha_k \underset{\Omega}{\to} \alpha} \jacl(\alpha_k)} ,
\]
where $\conv$ is the convex hull and $\alpha_k \underset{\Omega}{\to} \alpha$ is a shorthand notation for $\alpha_k \to \alpha$ and $\alpha_k \in \Omega$. It is known, see \cite[Proposition~6.2.2]{Clarke}, that $\partial^C \lfun(\alpha)$ is a non-empty convex compact subset of $\RR^r$.

Semi-smooth functions (see~\cite[Definition~7.4.2]{FacchineiPang03}) are precisely (locally) Lipschitz continuous functions for which the Clarke Jacobians define a legitimate Newton approximation scheme in the sense of~\cite[Definition~7.2.2]{FacchineiPang03}. Here, we will even consider an inexact semi-smooth Newton method which is detailed in Algorithm~\ref{alg:semismoothnewton}.

\begin{algorithm}[H]
\caption{\label{alg:semismoothnewton}Semi-smooth Newton to solve $\lfun(\alpha)=0$}
\begin{algorithmic}[1]
	\REQUIRE A point $\alpha_0 \in \RR^n$.
	\FORALL{$k=0,1,2,\ldots$}
	\IF{$\lfun(\alpha_k) = 0$} stop. \label{alg:finitestep}
	\ELSE
	\STATE Select $G_k \in \partial^C \lfun(\alpha_k)$, compute $\alpha_{k+1}$ such that
	\[
	\lfun(\alpha_k) + G_k(\alpha_{k+1} - \alpha_k) = e_k,
	\]
	where $e_k \in \RR^r$ is an error term satisfying $\norm{e_k} \le \eta_k \norm{G_k}$ and $\eta_k \geq 0$.
	\ENDIF
	\ENDFOR
\end{algorithmic}  
\end{algorithm}

It remains now to identify a broad class of convex functions $\f$ to which Algorithm~\ref{alg:semismoothnewton} applies. A rich family will be provided by semi-algebraic functions, i.e., functions whose graph is defined by some Boolean combination of real polynomial equations and inequalities~\cite{coste2002intro}. An even more general family is that of definable functions on an o-minimal structure over $\RR$, which corresponds in some sense to an axiomatization of some of the prominent geometrical properties of semi-algebraic geometry~\cite{vandenDriesMiller96,coste1999omin}. A slightly more general notion is that of a tame function, which is a function whose graph has a definable intersection with every bounded box~\cite[Definition~2]{BolteSN09}. Given the variety of optimization problems that can be formulated within the framework of o-minimal structures, our convergence result for Algorithm~\ref{alg:semismoothnewton} will be stated for tame functions.

\begin{proposition}[Convergence of Algorithm~\ref{alg:semismoothnewton}]
\label{prop:generaldh-rank-r}
Consider the situation of Theorem~\ref{theo:proxVrankr}, where $\f$ is in addition a tame function. Then $\lfun$ is semi-smooth and all elements of $\partial^C \lfun(\alpha^\star)$ are non-singular. In turn there exists $\overline{\eta}$ such that if $\eta_k \leq \overline{\eta}$ for every $k$, there exists a neighborhood of $\alpha^\star$ such that for all $\alpha_0$ in that neighborhood, the sequence generated by Algorithm~\ref{alg:semismoothnewton} is well-defined and converges to $\alpha^\star$ linearly. If $\eta_k \to 0$, the convergence is superlinear.

In particular, if $\f$ is semi-algebraic and $e_k=0$, then there exists a rational number $q > 0$ such that
\[
\norm{\alpha_k - \alpha^\star} = O\parenth{\exp(-(1+q)^k)} .
\]
\end{proposition}

The proof is in Section~\ref{appendix:prop:generaldh-rank-r}.

Proposition~\ref{prop:generaldh-rank-r} provides a remarkably fast local convergence guarantee of Algorithm~\ref{alg:semismoothnewton} to find the unique zero of $\lfun$ in \eqref{eq:roothgenr} provided one start sufficiently close to that zero. If this requirement is not met, the convergence of the algorithm is not ensured anymore. However we can say that $\norm{\alpha^\star} \leq \beta$, where the radius $\beta$ can be easily estimated from \eqref{eq:minalpha-rank-r}. For instance, for the metric $V = P+Q$, by strong convexity of modulus $c=1$ (see~Theorem~\ref{theo:proxVrankr}), we have
\[
\norm{\alpha^\star}^2/2 \leq \env{{\bpa{\f^*\circ \matP^{1/2}}}}{{1}}(\matP^{1/2}x) - \inf \env{{\bpa{\f^*\circ \matP^{1/2}}}}{{1}} + \frac 12 \norm{x}^2_{\matQ^+} .
\]
If $0 \in \dom(\f)$, we have the bound, valid for any $z \in \RR^N$,
\[
-\f(0) = \inf(\f^*) \le \f^* \circ \matP^{1/2}(p) \le \tfrac{1}{2}\norm{z-p}^2 + \f^* \circ \matP^{1/2}(p) = \env{{\bpa{\f^*\circ \matP^{1/2}}}}{{1}}(z) .
\] 
where we denoted $p=\prox_{\f^* \circ \matP^{1/2}}(z)$. Thus, setting $\beta=\env{{\bpa{\f^*\circ \matP^{1/2}}}}{{1}}(\matP^{1/2}x)+\frac 12 \norm{x}^2_{\matQ^+}+h(0)$, one can initialize Algorithm~\ref{alg:semismoothnewton} with $\alpha_0$ in the ball of radius $\beta$. An alternative way is to run e.g. an accelerated gradient descent (Nesterov or FISTA), initialized with such $\alpha_0$, a few iterations on the  strongly smooth problem~\eqref{eq:minalpha-rank-r} in $\RR^r$ (recall $r \ll N$), and use the final iterate as an initialization of Algorithm~\ref{alg:semismoothnewton}. Note that accelerated (FISTA-type) gradient descent is linearly convergent with the optimal rate $1-\sqrt{\mathrm{cond}^{-1}}$, where $\mathrm{cond}=(1+\opnorm{\matP^{-1/2}U}^2)/c$ is the condition number of problem~\eqref{eq:minalpha-rank-r} (see Theorem~\ref{theo:proxVrankr}).

}

\subsection{Diagonal $\pm$ rank-1 metric} \label{sec:prox-diag-pm-rank-1}

Here we deal with metrics of the form $V = D \pm uu^\top \in \sdp(N)$ which will be at the heart of our quasi-Newton splitting algorithm, where $D$ is diagonal with (strictly) positive diagonal elements $d_i$, and $u \in \RR^N$. 

\subsubsection{General case}
We start with the general case where $\f$ is any function in $\Gamma_0(\Hm)$.
\begin{theorem}[Proximity operator for a diagonal $\pm$ rank-1 metric]
\label{theo:proxVrank1}
Let $\f \in \Gamma_0(\Hm)$. Then,
\be
\label{eq:proxVgen}
\prox^V_\f(x) = D^{-1/2} \circ \prox_{\f \circ D^{-1/2}}\circ D^{1/2}(x \mp \alpha^\star D^{-1}u) ~,
\ee
where $\alpha^\star$ is the unique root of
\be  
\label{eq:roothgen}
\lfun(\alpha) = \pds{u}{x - D^{-1/2} \circ \prox_{\f \circ D^{-1/2}}\circ D^{1/2}(x \mp \alpha D^{-1} u)} + \alpha ~,
\ee
which is a strongly increasing and Lipschitz continuous function on $\RR$ with Lipschitz constant $1+\sum_i u_i^2/d_i$.
\end{theorem}

Theorem~\ref{theo:proxVrank1} is a specialization of Theorem~\ref{theo:proxVrankr}.

\begin{remark}
{~}
\begin{itemize}
\item \JF{There is a large class of functions for which $\prox_{\f \circ D^{-1/2}}$ can be computed either exactly or efficiently. The case of a separable function $\f$ will be considered in  Section~\ref{subsec:separable}}, 
\SB{ but the computation is efficient even for many non-separable functions such as the indicator of the simplex and the $\max$ function (see Table~\ref{tab:listOfFcn}), and many others.}
\item It is of course straightforward to compute $\prox^V_{\f^*}$ from $\prox^{V^{-1}}_\f$ either using Theorem~\ref{theo:proxVrank1}, or using this theorem together with Lemma~\ref{lem:moreauidentV} and the Sherman-Morrison inversion lemma. Indeed, when $V=D \pm uu^\top$ then $V^{-1}=D^{-1} \mp vv^\top$, where $v=D^{-1}u/\sqrt{1\pm\sum_{i}\tfrac{u_i^2}{d_i}}$. 
\item The formula for the inverse is also important for the forward step \eqref{eq:fbsr1} in Algorithm~\ref{alg:main}.
\item The theory developed in \cite{BF12} accounts for the proximity operator w.r.t. a metric $V=D+u u^\top$ (diagonal $+$ rank-1), which is extended here to the case $V=D\pm u u^\top$. Karimi and Vavasis \cite{KV17} developed an algorithm for solving the proximity operator of the (separable) $\ell_1$-norm with respect to a metric $V=D-u u^\top$, which is not covered in \cite{BF12}. The results in Theorems~\ref{theo:proxVrankr} and~\ref{theo:proxVrank1} are far-reaching generalizations that formalize the algorithmic procedure in \cite{KV17}.
\end{itemize}
\end{remark}

\SB{
\subsubsection{Bisection search} \label{subsec:rank1bisection}

We here discuss solving \eqref{eq:roothgen} via the bisection method in Algorithm~\ref{alg:bisection}, since this will allow us to produce a global complexity bound. The key tool is a bound on the values of $\alpha$ given by the following proposition which is valid even if $P$ is not diagonal.
    
\begin{proposition} \label{prop:bound}
For $r=1$, the root $\alpha^\star$ of~\eqref{eq:roothgen} lies in the set $[-\beta,\beta]$ where
\begin{equation} \label{eq:bound}
\beta = \|u\|\cdot\left( 2\|x\| + \norm{\prox^V_{\f}(0)} \right)
\end{equation}
where $\prox^V_{\f}(0)$ is a constant (e.g., it is zero if $0\in\argmin(\f)$, as it is for all positively homogeneous functions).
\end{proposition}

The proof is in Section \ref{appendix:prop:bound}.

\begin{proposition}[Convergence of Algorithm~\ref{alg:bisection}]  \label{prop:bisectionConverges}
For any $\epsilon>0$, Algorithm~\ref{alg:bisection} will produce a point $\alpha$ such that $|\alpha - \alpha^\star| \le \epsilon$ in $\log_2 \left(\epsilon/(2c\beta)\right)$ steps, where $\beta$ is as in~\eqref{eq:bound}, and $c$ is the strong monotonicity modulus given in Theorem~\ref{theo:proxVrankr}.
\end{proposition}

The proof of the above proposition is immediate, since $\lfun$ is a strongly monotone operator and one-dimensional, hence $\lfun$ is a monotonically increasing function, and thus the bisection method works. Strong monotonicity implies that for all $\alpha \in \RR$, $|\lfun(\alpha)| \geq c |\alpha-\alpha^\star|$. 

The bisection procedure is outlined in Algorithm \ref{alg:bisection}; note that later we will provide Algorithm \ref{alg:root-pl-sep} which is a specialization of bisection to a special class of functions $h$ for which we can find the root with zero error (assuming exact arithmetic). Note that a variant of Proposition~\ref{prop:bound} holds when $r > 1$ (see end of Section~\ref{subsec:semi-smooth-Newton}), but there is no analog to the bisection method in dimension $r>1$ since there is no total order.

    
\begin{algorithm}
   \caption{\label{alg:bisection}Bisection method to solve $\lfun(\alpha)=0$ when $r=1$}
   \begin{algorithmic}[1]
       \REQUIRE Tolerance $\epsilon>0$
       \STATE Compute the bound $\beta$ from \eqref{eq:bound}, and set $k=0$
       \STATE Set $\alpha_{-} = -\beta$ and $\alpha_{+}=\beta$
       \FORALL{$k=0,1,2,\ldots$}
           
            \STATE Set $\alpha_{k} = \frac{1}{2}\left( \alpha_- + \alpha_+\right)$
        \IF{ $\lfun(\alpha_k) > 0$ }
            \STATE $\alpha_+ \gets \alpha_k$
        \ELSE
            \STATE $\alpha_- \gets \alpha_k$
        \ENDIF
        \IF{$k>1$ and $|\alpha_k - \alpha_{k-1}|<\epsilon$} 
            \RETURN $\ \alpha_k$
        \ENDIF
       \ENDFOR
        \end{algorithmic}  
    \end{algorithm}

}

\subsubsection{Separable case}\label{subsec:separable}
The following corollary states that the proximity operator takes an even more convenient form when $\f$ is separable. It is a specialization of Theorem~\ref{theo:proxVrank1}.
\begin{corollary}[Proximity operator for a diagonal $\pm$ rank-1 metric for separable functions]
\label{cor:proxVseprank1}
Assume that $\f \in \Gamma_0(\Hm)$ is separable, i.e. $\f(x)=\sum_{i=1}^N \f_i(x_i)$, and $V = D \pm uu^\top \in \sdp(N)$, where $D$ is diagonal with (strictly) positive diagonal elements $d_i$, and $u \in \RR^N$. Then
\be
\label{eq:proxVsep}
 \prox^V_{\f}(x)  = \parenth{\prox_{\f_i/d_i}(x_i \mp \alpha^\star u_i/d_i) }_{i=1}^N~,
\ee
where $\alpha^\star$ is the unique root of
\be
\label{eq:roothsep}
\lfun(\alpha) = \pds{u}{x - \parenth{\prox_{\f_i/d_i}(x_i \mp \alpha u_i/d_i)}_{i=1}^N} + \alpha ~,
\ee
which is a Lipschitz continuous and strongly increasing function on $\RR$.
\end{corollary}

In particular, when the proximity operator of each ${\f}_i$ is piecewise affine, we get the following.
\begin{proposition}
\label{prop:proxVseprank1}
Consider the situation of Corollary~\ref{cor:proxVseprank1}. Assume that for $1 \leq i \leq N$, $\prox_{{\f}_i/d_i}$ is piecewise affine on $\RR$ with $k_i \geq 1$ segments, \ie 
\begin{equation}
\label{eq:proxpieceaffine}
  \prox_{\f_i/d_i}(x_i) =
  \begin{cases}
  a^i_0 x_i + b^i_0, &\text{if}\ x_i \leq t^i_{1}~; \\
  a^i_j x_i + b^i_j, &\text{if}\ t^i_j \leq x_i \leq t^i_{j+1},\ j \in \{1,\ldots,k_i\}~; \\
  a^i_{k_i+1} x_i + b^i_{k_i+1}, &\text{if}\ t^i_{k_i+1}\leq x_i~,
  \end{cases}
\end{equation}
for some $a^i_j,b^i_j\in \R$, and define $t_0^i:=-\infty$ and $t^i_{k_i+2}:=+\infty$. Then $\prox^V_\f(x)$ can be obtained exactly using Algorithm~\ref{alg:root-pl-sep} with binary search for Step~\ref{alg:root-pl-sep-step-rootfinding} in $O(K\log(K))$ steps where $K=\sum_{i=1}^N k_i$.
\end{proposition}

The proof is in Section~\ref{appendix:prop:proxVseprank1}.

Using Proposition \ref{prop:proxVseprank1}, we derive Algorithm~\ref{alg:root-pl-sep}.
\begin{algorithm}[H]
    \caption{\label{alg:root-pl-sep}Exact root finding algorithm for piecewise affine separable proximity operators}
\begin{algorithmic}[1]
  \REQUIRE Piecewise affine proximity operator $\prox_{\f_i/d_i}(x_i)$, $i=1,\ldots, N$, as defined in Proposition~\ref{prop:proxVseprank1}.
  \STATE Sort $\widetilde{\bm\theta}:=\bigcup_{i=1}^N \{ \pm \tfrac{d_i}{u_i}(x_i - t^i_j)  \,:\, j=1,\ldots,k_i \}\subset\R$ into a list $\bm\theta\in\R^{k^\prime}$ with $k^\prime\leq K$. \label{alg:root-pl-sep-sort-bpts}
  \STATE Set $\overline{\bm\theta} := [-\infty, \bm\theta_1,\ldots,\bm\theta_{k^\prime},+\infty]$.
  \STATE Via the bisection method, detect the interval $[\theta_-,\theta_+)$ with adjacent $\theta_-,\theta_+\in\overline{\bm\theta}$ that contains the root of $\lfun(\alpha)$. \label{alg:root-pl-sep-step-rootfinding}
  \STATE Compute the root $\alpha^\star=-b/a$ where $a$ and $b$ are determined as follows:
  \STATE For all $i=1,\ldots,N$, define $j_i\in\{0,\ldots,k_i+1\}$ such that $t^i_{j_i}\leq \theta_-<\theta_+\leq  t^i_{j_i+1}$, and compute \label{alg:root-pl-sep-aff-accum} 
   \[
     a:= 1\pm\sum_{i=1}^N  a^i_{j_i} u_i^2/d_i
     \quad\text{and}\quad 
     b:=\sum_{i=1}^N u_i ((1- a^i_{j_i}) x_i - b^i_{j_i})\,.
   \]
\end{algorithmic}  
\end{algorithm}

Some remarks are in order.

\begin{remark} 
{~} 
\begin{itemize}
  \item The sign ``$\pm$'' in Algorithm~\ref{alg:root-pl-sep} refers to the two cases of $V=D\pm u u^\top$ from Corollary~\ref{cor:proxVseprank1}.
  \item Since \eqref{eq:roothsep} is piecewise affine, $a,b$ in Step~\ref{alg:root-pl-sep-aff-accum} for the interval $[\theta_-,\theta_+)$, can be determined by 
  \[
    a=\frac{\lfun(\theta_+^\prime)-\lfun(\theta_-^\prime)}{\theta_+^\prime-\theta_-^\prime}
    \quad\text{and}\quad
    b = \lfun(\theta_-^\prime)~\,
  \]
  where $\theta_-\leq \theta_-^\prime < \theta_+^\prime\leq \theta_+$ and $-\infty<\theta_-^\prime$ and $\theta_+^\prime<+\infty$. (The usage of ``$\theta^\prime$'' avoids ``$\lfun(-\infty)$''.)
\end{itemize}
\end{remark}

\begin{remark}
\label{prop:proxVseprank1linear}
{~} 
\begin{itemize}
\item The bulk of complexity in Proposition~\ref{prop:proxVseprank1} lies in locating the appropriate breakpoints. This can be achieved straightforwardly by sorting followed by a bisection search, as advocated, whose worst-case computational complexity is nearly linear in $N$ up to a logarithmic factor. The log term can theoretically be removed by replacing sorting with a median-search-like procedure whose expected complexity is linear.
\item The above computational cost can be reduced in many situations by exploiting, e.g., symmetry of the $\f_i's$, identical functions, \etc This turns out to be the case for many functions of interest, \eg $\ell_1$-norm, indicator of the $\ell_\infty$-ball or the positive orthant, polyhedral seminorms, and many others; see examples hereafter. 
\item \JF{It goes without saying that Corollary~\ref{cor:proxVseprank1} can be extended to the ``block'' separable case (i.e.~separable in subsets of coordinates).}
\item \JF{It is important to stress the fact that the reasoning underlying Proposition~\ref{prop:proxVseprank1} and Algorithm~\ref{alg:root-pl-sep} extends to a much more general class of proximity operators $\prox_{\f_i}$, hence functions $f_i \in \Gamma_0(\RR)$. Indeed, assume that $\f_i$ is definable (see Section~\ref{subsec:semi-smooth-Newton} for details on definable functions). Thus arguing as in the proof of Proposition~\ref{prop:generaldh-rank-r}, we have that $\prox_{\f_i}$ is also definable. It then follows from the monotonicity lemma~\cite[Theorem~4.1]{vandenDriesMiller96} that for any $k \in \NN$, one can always find a finite partition $(t^i_j)_{1 \leq j \leq k_i}$ into $k_i$ disjoint intervals such that $\prox_{\f_i}$ restricted to each nontrivial interval is $C^k$ and strictly increasing or constant. With such a partition, the right-hand side of \eqref{eq:proxpieceaffine}~may be non-linear in $x_i$ but $C^k$ and increasing on the corresponding open interval. Consequently, the first three steps of Algorithm~\ref{alg:root-pl-sep}, which consist in locating the appropriate interval $[\theta_-,\theta_+)$ that contains the unique root $\alpha^\star$, remain unchanged. If $\alpha^\star \neq \theta_{\pm}$, only step~\ref{alg:root-pl-sep-aff-accum}, which computes $\alpha^\star$, has to be changed to any root finding method of a one-dimensional non-linear $C^k$ smooth function on $(\theta_-,\theta_+)$. For instance, we have shown that $\alpha^\star$ is a non-degenerate root ($\lfun$ is strictly increasing). Therefore, if $k=2$, then $\lfun \in C^2((\theta_-,\theta_+))$, and a natural root-finding scheme would be the Newton method which provides local quadratic convergence to $\alpha^\star$. More generally, if $k \geq 2$, local higher order convergence rate can be obtained with the Householder's class of methods.}
\item \JF{In view of the previous two remarks, the case of the $\ell_1-\ell_2$ norm, which is popularly used to promote group sparsity, can be handled by our framework. This example will be considered in more detail in Section~\ref{sec:examples}.}
\end{itemize}
\end{remark}

\subsubsection{Examples} \label{sec:examples}
Many functions can be handled very efficiently using our results above. For instance, Table~\ref{tab:listOfFcn} summarizes a few of them where we can obtain either an exact answer by sorting when possible, or else by minimizing w.r.t. to a scalar variable (\ie finding the unique root of \eqref{eq:roothgen}).

\begin{table}
    \centering
    \footnotesize
    \begin{tabular}{ll}
        \toprule
        Function $\f$ & Method \\
        \midrule
$\ell_1$-norm (separable) 	& exact with sorting \\
Hinge (separable)		& exact with sorting \\
Box constraint (separable) 	& exact with sorting \\
$\ell_\infty$-ball (separable)	& exact with sorting \\
Positivity constraint (separable) & exact with sorting \\
$\ell_1-\ell_2$ (block-separable)& sort and root finding \\
Affine constraint (nonseparable)& closed-form \\
$\ell_1$-ball (nonseparable)	& root-finding and $\prox_{\f \circ D^{-1/2}}$ costs a sort \\
$\ell_\infty$-norm (nonseparable) & from projector on the $\ell_1$-ball by Moreau-identity \\
Simplex (nonseparable)		& root-finding and $\prox_{\f \circ D^{-1/2}}$ costs a sort \\
$\max$ function (nonseparable)	& from projector on the simplex by Moreau-identity \\
        \bottomrule
    \end{tabular}
    \caption{A few examples of functions which have efficiently computable proximity operators in the metric $V=D \pm uu^\top$.}
    \label{tab:listOfFcn}
\end{table}

\paragraph*{Affine constraint} We start with a case where the proximity operator in the diagonal $\pm$ rank 1 metric has a closed-form expression. Consider the case where $\f=\indic_{\{x: A x = b\}}$. We then immediately get
\[
\prox_{\f \circ D^{-1/2}}(z) = z + Y^+(b - Y z) = \Pi z + c 
\]
where $Y=AD^{-1/2}$, $\Pi$ is the projector on $\Ker(Y)=D^{1/2}\Ker(A)$, and $c=Y^+ b$. After simple algebra, it follows from Theorem~\ref{theo:proxVrank1}, that the unique root of $\lfun$ in this case is
\[
\alpha^\star = \frac{\pds{u}{D^{-1/2}\parenth{c - (\Id - \Pi)D^{1/2}x}}}{1 \pm \pds{u}{D^{-1/2}\Pi D^{-1/2}u}} ~.
\]

\paragraph*{Positive orthant} We now put Proposition~\ref{prop:proxVseprank1} on a more concrete footing by explicitly covering the case when $h$ represents non-negativity constraints. 
 Consider $V=D+uu^\top$ and $\f = \indic_{\{x:\,x\ge 0\}}$. We will calculate
\be  \label{eq:nnlsEx}
\prox_{\f}^{V^{-1}}(x) = \argmin_{ y \ge 0} \frac{1}{2}\|y-x\|_{V^{-1}}^2
\ee

We use the fact that the projector on the positive orthant is separable with components $\parenth{x_i}_+ \eqdef \max(0,x_i)$, i.e. a piecewise affine function.  Define the scalar $\alpha = u^\top\lambda$. 
Let $\lambda_i^{(\alpha)} \eqdef \parenth{-(x_i + \alpha u_i)/d_i}_+$, so we search for a value of $\alpha$ such that $\alpha = u^\top\lambda^{(\alpha)}$, or in other words,
a root of $\lfun(\alpha) = \alpha -  u^\top\lambda^{(\alpha)}$.

Define $\hat{\alpha}_i$ to be the sorted values of $(-x_i/u_i)$, so 
we see that $\lfun$ is linear in the regions $[\hat{\alpha}_i,\hat{\alpha}_{i+1}]$
and so it is trivial to check if $\lfun$ has a root in this region.
Thus the problem is reduced to finding the correct region $i$,
which can be done efficiently by a bisection search over $\log_2(n)$ values
of $i$ since $\lfun$ is monotonic. To see that $\lfun$ is monotonic, we write
it as 
$$ 
\lfun(\alpha) = \alpha + \sum_{i=1}^N \left( (u_ix_i + \alpha u_i^2 )/d_i \right) \chi_i(\alpha)
$$
where $\chi_i(\alpha)$ encodes the positivity constraint in the argument of $\parenth{\cdot}_+$ and is thus either $0$ or $1$, hence the slope is always positive.

\paragraph*{$\ell_1-\ell_2$ norm} Let $\mathscr{B}$ be a uniform disjoint partition of $\{1,\ldots,N\}$, i.e. $\bigcup_{b \in \mathscr{B}}=\{1,\ldots,n\}$ and $b \cap b' = \emptyset$ for all $b \neq b' \in \mathscr{B}$. The $\ell_1-\ell_2$ norm of $x$ is 
\begin{equation} \label{eq:l1-l2-sparsity-term}
\norm{x}_{1,2} = \sum_{b \in \mathscr{B}} \norm{x_b}
\end{equation}
where $x_b$ is the subvector of $x$ indexed by block $b$.

Without of loss of generality, we assume that all blocks have the same size, and we consider the metric $V=D+uu^\top$, where the diagonal matrix $D$ is constant on each block $b$. We now detail how to compute the proximity operator in $\Hm_V$ of $\f=\lambda\norm{\cdot}_{1,2}$, $\lambda > 0$. For this, we will exploit Theorem~\ref{theo:proxVrank1} and the expression of $\prox_{\f \circ D^{-1/2}}$, i.e. block soft-thresholding. The latter gives
\[
\parenth{D^{-1/2}\prox_{\f \circ D^{-1/2}}(D^{1/2}x)}_b = \parenth{\prox_{\f \circ D^{-1}}(x)}_b = \parenth{1-\frac{\lambda}{d_b\norm{x_b}}}_+x_b, \qquad \forall b \in \mathscr{B} ~,
\]
where $d_b$ is the diagonal entry of $D$ shared by block $b$. This then entails that
\begin{equation*}
\begin{split}
\lfun(\alpha) = \pds{x}{u} + \alpha - \sum_{b \in \mathscr{I}(\alpha)}\parenth{\parenth{1-\frac{\lambda}{d_b\norm{x_b - \alpha u_b/d_b}}}\parenth{\pds{x_b}{u_b}-\alpha\norm{u_b}^2/d_b}} ~,
\end{split}
\end{equation*}
where $\mathscr{I}(\alpha)=\acc{b \in \mathscr{B}: \norm{x_b - \alpha u_b/d_b} \geq \lambda/d_b}$. This is a piecewise smooth function, with breakpoints at the values of $\alpha$ where the active support $\mathscr{I}(\alpha)$ changes. To compute the root of $\alpha$, it is sufficient to locate the two breakpoints where $\lfun$ changes sign, and then run a fast root-finding algorithm (\eg Newton's method) on this interval where $\alpha$ is actually $C^\infty$. Denote $N_{\mathscr{B}}=\lfloor N/\abs{b} \rfloor$ the number of blocks. There are at most $2 N_{\mathscr{B}}$ breakpoints, and these correspond to the two real roots of $N_{\mathscr{B}}$ univariate quadratic polynomials, 
 each corresponding to
\[
\norm{d_b x_b - \alpha u_b}^2 = \alpha^2\norm{u_b}^2 - 2\alpha d_b\pds{x_b}{u_b} + d_b^2\norm{x_b}^2 = \lambda^2 ~.
\] 
Sorting these roots costs at most $O(N_{\mathscr{B}}\log N_{\mathscr{B}})$. To locate the breakpoints, a simple procedure is a bisection search on the sorted values, and each step necessitates to evaluate $\lfun$. This search also costs at most $O(N_{\mathscr{B}}\log N_{\mathscr{B}})$ operations (observe that all inner products and norms in $\lfun$ can be computed once for all). In summary, locating the interval of breakpoints containing the root takes $O(N_{\mathscr{B}}\log N_{\mathscr{B}})$ operations, though we believe this complexity could be made linear in $N_{\mathscr{B}}$ with an extra effort.


\section{A SR1 forward--backward algorithm}
\label{sec:SR1}

\subsection{Metric construction}
Following the conventional quasi-Newton notation,
we let $B$ denote an approximation to the Hessian of $f$ and $H$ denote
an approximation to the inverse Hessian. 
All quasi-Newton methods update an approximation to the (inverse) Hessian
that satisfies the \emph{secant condition}:
\begin{equation}\label{eq:secant}
    H_k y_k = s_k, \quad\text{where}\quad y_k = \nabla f(x_k) - \nabla f(x_{k-1}), \quad s_k = x_k - x_{k-1}.
\end{equation}

Algorithm~\ref{alg:main} follows the SR1 method~\cite{SR1}, which uses a rank-1 update
to the inverse Hessian approximation at every step.  The SR1 method
is perhaps less well-known than BFGS, but it has the crucial property
that updates are rank-1, rather than rank-2,
and it is described ``[SR1] has now taken its place alongside the BFGS method as the pre-eminent updating formula.''\cite{GouldLectureNotes}.

We propose two important modifications to SR1. The first is to use limited-memory,
as is commonly done with BFGS. In particular, we use zero-memory, which means
that at every iteration, a new diagonal plus rank-one matrix is formed.
The other modification is to extend the SR1 method to the general setting
of minimizing $f + h$ where $f$ is smooth but $h$ need not be smooth;
this further generalizes the case when $h$ is an indicator function of a
convex set.  Every step of the algorithm replaces $f$ with a quadratic approximation, and keeps $h$ unchanged.
Because $h$ is left unchanged, the subgradient of $h$ is used in an \emph{implicit} manner,
in comparison to methods such as~\cite{YuVishwanathan10} that use an approximation
to $h$ as well and therefore take an \emph{explicit} subgradient step.

\begin{algorithm}[h]
    \caption{Sub-routine to compute the approximate inverse Hessian $H_k$, 0SR1 variant \label{alg:SR1} }
\begin{algorithmic}[1]
    \REQUIRE $k, s_k, y_k$ as in \eqref{eq:secant}; and $0< \gamma < 1, \; 0 < \tau_\text{min} <\tau_\text{max} $ 
\IF{$k=1$} 
    \STATE $H_0 \leftarrow \tau \Id$ where $\tau > 0$ is arbitrary
    \STATE $u_k \leftarrow 0$
\ELSE
    \STATE   $ \BBa \leftarrow \frac{ \pds{s_k}{ y_k } }{ \norm{y_k}^2 } $
    \hfill \COMMENT{Barzilai--Borwein step length}
    \STATE  Project $\BBa$ onto $[\tau_\text{min},\tau_\text{max}]$
  \STATE $H_0 \leftarrow \gamma \BBa \Id$
  \IF{ $\pds{s_k - H_0 y_k }{y_k} \le 10^{-8} \|y_k\|_2 \|s_k - H_0 y_k \|_2 $ }
        \STATE $u_k \leftarrow 0$  \hfill \COMMENT{Skip the quasi-Newton update}
  \ELSE
  \STATE $u_k \leftarrow (s_k - H_0 y_k)/\sqrt{\pds{s_k - H_0 y_k }{y_k}})$. 
  \ENDIF
 \ENDIF
 \RETURN \ $H_k = H_0 + u_ku_k^\top$ \hfill \COMMENT{$B_k = H_k^{-1}$ can be computed via the Sherman-Morrison formula}
\end{algorithmic}  
\end{algorithm}

\newcommand{\gk}{\nabla f(x_k) }

\paragraph{Choosing $H_0$} \label{sec:h0}
In our experience, the choice of $H_0$ is best if scaled with a Barzilai--Borwein
spectral step length
\begin{equation}
\BBa=\pds{s_k}{y_k}/\pds{y_k}{y_k}
    \label{eq:BBa}
\end{equation}
(we call it $\BBa$ to distinguish it from the other Barzilai--Borwein step size
$\BBb = \linebreak \pds{s_k}{s_k}/\pds{s_k}{y_k} \ge \BBa$).

In SR1 methods, the quantity $ \pds{s_k - H_0 y_k}{y_k}$ must be positive
in order to have a well-defined update for $u_k$. The update is:
\begin{equation}
    H_k = H_0 + u_k u_k^\top,\quad u_k = (s_k - H_0 y_k )/\sqrt{ \pds{s_k - H_0 y_k}{y_k} }.
    \label{eq:Hk}
\end{equation}
For this reason, we choose $H_0 = \gamma \BBa \Id$ with $0 < \gamma < 1$,
and thus $  0 \le \pds{s_k - H_0 y_k}{y_k} = (1-\gamma)\pds{s_k}{y_k}$.
If $\pds{s_k}{y_k}=0$,
then there is no symmetric rank-one update that satisfies the secant condition.
The inequality $\pds{s_k}{y_k} > 0$ is the \emph{curvature condition},
and it is guaranteed for all strictly convex objectives. 
Following the recommendation in~\cite{NocedalWright}, we skip updates
whenever $\pds{s_k}{y_k}$ cannot be guaranteed to be non-zero
given standard floating-point precision.

A value of $\gamma=0.8$ works well in most situations.
We have tested picking $\gamma$ adaptively,
as well as trying $H_0$ to be non-constant on the diagonal, but found no consistent improvements.

\subsection{Convergence analysis}
For our convergence analysis, we naturally assume that $f$ is also $\mu$-strongly convex. \JF{This assumption is standard for Newton and quasi-Newton methods if one wants to get provable convergence guarantees. Indeed, one has to assume some non-singularity assumption for the iterates to be well-defined. We can make our strong convexity assumption hold only locally around a minimizer, but our guarantees will also become of local nature. The strong convexity assumption can be weakened to restricted strong convexity when $\f=\iota_S$, where $S \subset \RR^N$ is a linear subspace. In this case, problem~\eqref{eq:minP} is equivalent to 
\[
\min_{x \in S} f \circ \proj_S(x) .
\]
Thus, since $P=(\gamma \BBa)^{-1} \Id$ for the 0SR1 and 0BFGS metrics, it follows from \eqref{eq:proxVgenr} that $\prox_{\kappa_k \f}^{B_k}(x) \in S$. Hence, from~\eqref{eq:fbsr1}, the quasi-Newton forward-backward sequence $\parenth{x_k}_{k \in \NN} \subset S$. In turn, the quasi-Newton vectors $s_k$ and $y_k$ belong to $S$, i.e., $\forall k \in \NN$
\[
s_k=x_k-x_{k-1} \in S \quad \text{and} \quad y_k=\proj_S(\nabla f(\proj_S(x_k))) - \proj_S(\nabla f(\proj_S(x_{k-1}))) \in S  .
\]
Now, assuming that $\f$ is strongly convex on $S$ and its gradient is Lipschitz on $S$, with constants $\mu_S$ and $L_S$, the bounds on the eigenvalues of matrices $H_k$ in Lemma~\ref{lem:hkbnds} and Lemma~\ref{lem:hkbnds-bfgs} hereafter will remain true with $(\mu,L)$ replaced by $(\mu_S,L_S)$. The convergence claims of Theorem~\ref{thm:convergence} and Theorem~\ref{thm:convergence-bfgs} will also hold with rates characterized by the condition number $L_S/\mu_S$ rather than $L/\mu$.}\\

The following lemma delivers useful uniform bounds on the eigenvalues of matrices $H_k$.

\begin{lemma}
\label{lem:hkbnds}
Suppose that $f$ is $\mu$-strongly convex and its gradient is $L$-Lipschitz. Then, $\forall k \geq 0$, $a \Id \preceq H_k \preceq b \Id$, $0 < a=\gamma L^{-1}$, $0 < b =\frac{(1+\gamma)\mu^{-1}-2\gamma L^{-1}}{1-\gamma}$.
\end{lemma}

The proof is in Section~\ref{appendix:lem:hkbnds}.
%
\begin{theorem}
\label{thm:convergence}
Suppose that $f$ is $\mu$-strongly convex and its gradient is $L$-Lipschitz. Let $a$ and $b$ be given as in Lemma~\ref{lem:hkbnds}. Assume that $0 < \underline{\step} \leq \stk \leq \overline{\step} < 2(Lb)^{-1}$. Let $\alpha=1-\tfrac{Lb\overline{\step}}{2}$ and $\eta=\tfrac{L}{2\gamma\mu\underline{\step}}$. Then, the sequence of iterates $(x_k)_{k \in \NN}$ of the 0SR1 forward--backward Algorithm~\ref{alg:main} with $t=1$ converge linearly to the unique minimizer $\xs$, i.e.
\[
\norm{\xk - \xs} \leq \sqrt{\frac{2\parenth{F(x_0) -F(\xs)}}{\mu}}\rho^{k/2} ~,
\]
where
\begin{equation*}
\rho = 
\begin{cases}
\rho_1 				& \text{ for } \alpha \in ]0,1/2[ \\
\min(\rho_1,\rho_2) 	& \text{ for } \alpha \in [1/2,1[
\end{cases} ~,
\end{equation*}
with
\begin{align*}
\rho_1 = 1 - \alpha\parenth{1-2\parenth{\sqrt{\eta^2+\eta}-\eta}} \quad \text{and} \quad
\rho_2 = 
\begin{cases}
2\eta ~ (\leq 1/2) 	& \text{ if } \eta \leq 1/4 \\
1 - \frac{1}{8\eta} 	& \text{ otherwise } ~.
\end{cases}
\end{align*}
\end{theorem}

The proof is in Section~\ref{appendix:thm:convergence}. 
Fig.~\ref{fig:rates} shows the phase diagram of the rate $\rho$ as a function of $\eta$ and $\alpha$.

Actually, this is the standard setting for Newton and quasi-Newton methods if one wants to get provable convergence guarantees. Indeed, one has to assume some non-singularity assumption for the iterates to be well-defined. We can make our strong convexity assumption holds only locally around a minimizer, but our guarantees will also become of local nature. 

\begin{remark}\label{rem:ratesr1}
    For a concrete example of the rates in Theorem \ref{thm:convergence}, choose $\gamma=1/2$ so that $a=1/(2L)$ and $b=3\mu^{-1} - 2L^{-1}$, and choose $\stk \equiv \overline{\step} = \underline{\step} = 1/(Lb)$.
    Thus $\alpha=1/2$. Let $c=L/\mu$ be the condition number of the problem. Then $\eta = c(3c-2)$, and so for large $c\gg 1$, we have $\eta \gg 1$ and via Taylor expansion we see that $\rho_1 - \rho_2 \rightarrow 0$
    as $\eta \rightarrow \infty$. In turn, the rate of linear convergence is $\rho \approx 1-1/(8\eta) \approx 1 - 1/(24 c^2)$. \POc{}{Although, this rate is apperently worse than, for example, the standard rate obtained for forward-backward, our numerical experiments demonstrate that the performance is significantly better than this worst case prediction. Unless the metric approximates second order information, which is not the case for our zero memory variant, we do not expect to improve the convergence rate. Possibly, a deep analysis might improve the constants appearing in the convergence rate estimate. However, the efficiency of our method comes from an ``optimal'' compromise between locally adapting the metric and a cheap computability of the update step.} 
\end{remark}

\begin{SCfigure}
\includegraphics[width=0.5\textwidth]{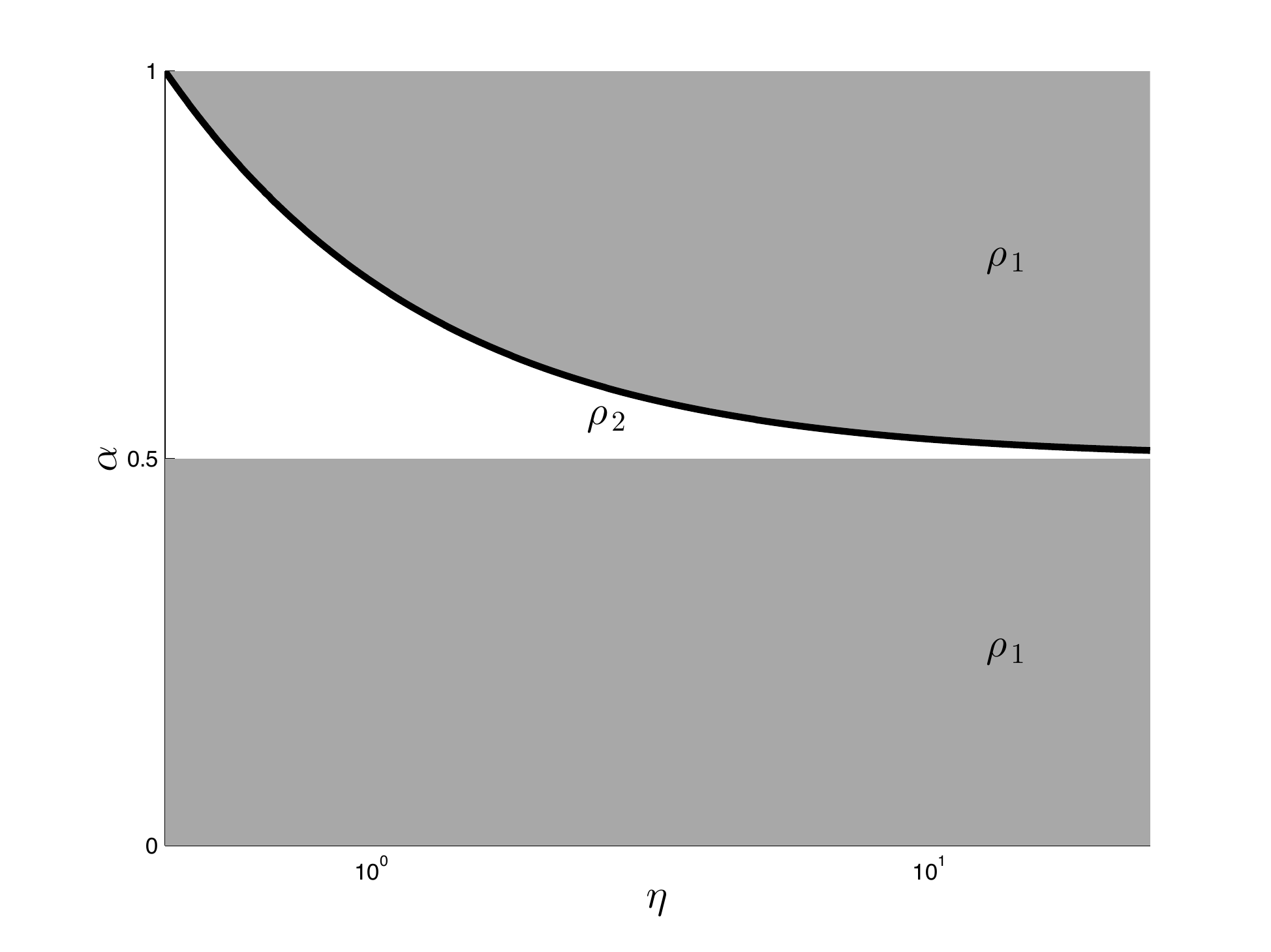}
\caption{Convergence rate as a function of the parameters $\eta$ and $\alpha$ (see Theorem~\ref{thm:convergence} for details).}
\label{fig:rates}
\end{SCfigure}


\section{L-BFGS forward--backward splitting}
\label{sec:BFGS}

In this section, we show how the extended theory for rank-$r$ modified proximity operators in Section~\ref{sec:rank-r-mod-prox} can be used for the efficient treatment of the more sophisticated L-BFGS method in our context of proximal quasi-Newton methods. We consider Algorithm~\ref{alg:main} where the metric construction is outlined in Section~\ref{subsec:L-BFGS-metric} following the notation in \cite{NocedalWright}. The proximity operator in \eqref{eq:fbsr1} will be of type ``diagonal $\pm$ rank-$2$''. 

\subsection{Metric construction}
\label{subsec:L-BFGS-metric}

Define
\[
    \rho_k = \frac{1}{y_k^\top s_k},\quad V_k = \Id - \rho_k y_ks_k^\top, ~~\text{with}~~ s_k = x_{k+1}-x_k,\quad y_k = \nabla f(x_{k+1}) - \nabla f(x_k)
\]
as in \eqref{eq:secant}. 
Store $\{s_i,y_i\}$ for $i=k-m,k-m-1,\ldots,k-1$.
Choose $H_k^0$ as before, e.g., $H_k^0 = \gamma \tau \Id$. Then the limited-memory BFGS (L-BFGS) quadratic approximation is 
\begin{align*}
    H_k &= ( V_{k-1}^\top \cdots V_{k-m}^\top) H_k^0 (V_{k-m}\cdots V_{k-1}) \\
        &\quad + \rho_{k-m}(V_{k-1}^\top \cdots V_{k-m+1}^\top)s_{k-m}s_{k-m}^\top(V_{k-m+1}\cdots V_{k-1} ) \\
        &\quad + \rho_{k-m+1}(V_{k-1}^\top \cdots V_{k-m+2}^\top)s_{k-m+1}s_{k-m+1}^\top(V_{k-m+2}\cdots V_{k-1} ) \\
        &\quad + \cdots + \rho_{k-1}s_{k-1}s_{k-1}^\top.
\end{align*}
In the classical (unconstrained) L-FBGS, the update is then  $x_{k+1} = x_k - \alpha_k H_k \nabla f_k$.

In the extreme low-memory case ($m=1$), we have
\[ H_{k+1} = V_k^\top H_k^0 V_k + \rho_k s_ks_k^\top \]
which gives us a 0-BFGS method. For this $m=1$ case and $\tau=\BBa$, writing $V$ for $V_k$ and so on, we can expand
\begin{equation}
\label{eq:BFGS-inv-Hess-formula-derivation}
  \begin{split}
    H_k =&\  V^\top H_k^0 V + \rho ss^\top \\
    =&\  (\Id-\rho sy^\top)(\gamma\tau\Id)(\Id-\rho ys^\top) + \rho ss^\top \\
    =&\ \gamma\tau( \Id - \rho(ys^\top+sy^\top)+\rho^2\|y\|^2ss^\top ) + \rho ss^\top \\
    \smat{\rho \|y\|^2 \tau = 1}
    =&\ \gamma\tau \Id + \rho(1+\gamma) \Big( ss^\top - \frac{\gamma\tau}{1+\gamma}(sy^\top + ys^\top) + \frac{\gamma^2 \tau^2}{(1+\gamma)^2} y y^\top \Big) - \rho \frac{\gamma^2 \tau^2}{1+\gamma} yy^\top \\
    =&\  \gamma\tau \Id + \rho(1+\gamma) \Big(\underbrace{s- \frac{\gamma\tau}{1+\gamma}y}_{=:u_\gamma}\Big) \Big(s- \frac{\gamma\tau}{1+\gamma}y\Big)^\top - \rho \frac{\gamma^2 \tau^2}{1+\gamma} yy^\top \,,
  \end{split}
\end{equation}
which shows that the inverse Hessian approximation is of type ``diagonal $+$ rank-1 $-$ rank-1'' with positive semi-definite rank-1 matrices. Note that we are free to choose $\gamma=1$, in which case the simpler expression follows:
\begin{equation} \label{eq:BFGS-inv-Hess-formula}
  H_k = \tau \Id + 2\rho \Big(s- \frac{\tau}{2}y\Big) \Big(s- \frac{\tau}{2}y\Big)^\top - \rho \frac{\tau^2}{2} yy^\top \,,
\end{equation}
Applying the Sherman--Morrison inversion lemma to this, we obtain the following approximation to the Hessian matrix $B_k=H_k^{-1}$:
\[
   B_k = B_k^0 - \frac{B_k^0ss^\top B_k^0}{s^\top B_k^0 s} + \frac{y y^\top}{y^\top s} 
     = \frac{1}{\gamma\tau}\Big(\Id - \frac{ss^\top}{s^\top s} + \gamma \tau \frac{yy^\top}{y^\top s}\Big) 
     \overset{(\tau=\BBa)}{=} 
     \frac{1}{\gamma\BBa}\Big(\Id - \frac{ss^\top}{s^\top s} + \gamma\frac{yy^\top}{y^\top y}\Big).
\]
The proximity operator with respect to this metric can be computed as shown in Corollary~\ref{cor:proxVrankr}. Only the evaluation of the simple proximity operator $\prox_{\f}^{B_0}$ is required. The main computational cost comes from the two dimensional root finding problem, which can be solved efficiently using semi-smooth Newton methods.

%

\subsection{Convergence analysis}

For the convergence analysis, we again assume that $f$ is also $\mu$-strongly convex. We start with a lemma which provides useful uniform bounds on the eigenvalues of matrices $H_k$.

\begin{lemma}
\label{lem:hkbnds-bfgs}
Suppose that $f$ is $\mu$-strongly convex and its gradient is $L$-Lipschitz. Then, $\forall k \geq 0$, $a \Id \preceq H_k \preceq b \Id$, $0 < a=\gamma/(1+\gamma) L^{-1}$, $0 < b = (1+2\gamma) \mu^{-1}  - \frac{(2+\gamma)\gamma}{1+\gamma} L^{-1}$.
\end{lemma}

The proof is in Section~\ref{appendix:lem:hkbnds-bfgs}.
%
\begin{theorem}
\label{thm:convergence-bfgs}
Suppose that $f$ is $\mu$-strongly convex and its gradient is $L$-Lipschitz. Let $\gamma>0$, and $a,b$ be given as in Lemma~\ref{lem:hkbnds-bfgs}. Assume that $0 < \underline{\step} \leq \stk \leq \overline{\step} < 2(Lb)^{-1}$. Let $\alpha=1-\tfrac{Lb\overline{\step}}{2}$ and $\eta=\tfrac{L}{2\gamma\mu\underline{\step}}$. Then, the sequence of iterates $(x_k)_{k \in \NN}$ of the L-BFGS forward--backward Algorithm~\ref{alg:LBFGS} (with $H_k$ as in \eqref{eq:BFGS-inv-Hess-formula-derivation}) with $t=1$ converges linearly to the unique minimizer $\xs$, i.e.
\[
\norm{\xk - \xs} \leq \sqrt{\frac{2\parenth{F(x_0) -F(\xs)}}{\mu}}\rho^{k/2} ~,
\]
where $\rho$ is as given in Theorem~\ref{thm:convergence}.
\end{theorem}

The proof is the same as that of Theorem~\ref{thm:convergence} (Section~\ref{appendix:thm:convergence}) by substituting the constants $a$ and $b$ in Lemma~\ref{lem:hkbnds} with those in Lemma~\ref{lem:hkbnds-bfgs}. Note that the phase diagram in Fig.~\ref{fig:rates} still applies, though the underlying constants are slightly changed.

\begin{remark}
    Let's again illustrate the rate in Theorem~\ref{thm:convergence-bfgs}. We choose $\gamma=1/2$ as in Remark~\ref{rem:ratesr1} so that $a=1/(3L)$, $b=2\mu^{-1} - 5/6L^{-1}$, and $\stk \equiv \overline{\step} = \underline{\step} = 1/(Lb)$.
    Thus $\alpha=1/2$ and $\eta = c(2c-5/6)$, where $c=L/\mu$ is the condition number of the problem. For large $c\gg 1$, the rate of linear convergence is $\rho \approx 1-1/(8\eta) \approx 1 - 1/(16 c^2)$, which is smaller than the one of 0SR1 in Remark~\ref{rem:ratesr1}.  
\end{remark}

\section{Numerical experiments and comparisons}
\label{sec:results}
In the spirit of reproducible research, and to record the exact algorithmic details, all code for experiments from this paper is available at {\small\url{https://github.com/stephenbeckr/zeroSR1/tree/master/paperExperiments}}.

\subsection{LASSO problem}\label{sec:num-exp-lasso}

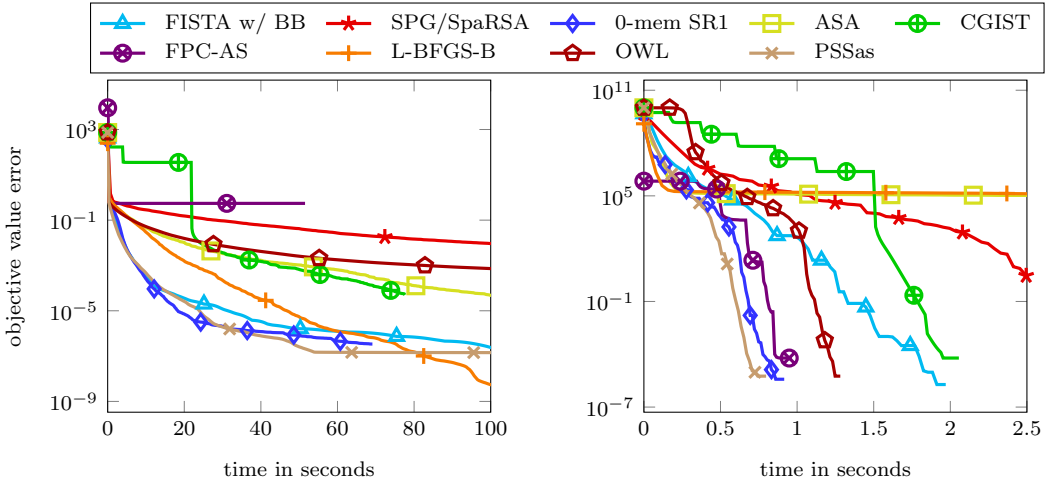
\begin{figure}[t]
  \begin{center}
  \ifcompilePGFfigs
    \begin{tikzpicture}
      \newcommand\markSize{3} 
      \begin{axis}[%
        width=0.44\linewidth,%
        height=6cm,%
        xmin=0,xmax=100,%
        ymode=log,
        xlabel={time in seconds},%
        ylabel={objective value error},%
        legend columns=5,
        legend cell align=left,
        legend style={at={(1.2,1.02)},anchor=south,font=\footnotesize,column sep=5pt}
        ]
      \addplot[very thick,cyan!80!white,solid,mark=triangle,mark repeat={100},mark size={\markSize},mark options={solid}] %
          table {figs/LASSO_test4_FISTAwithBB_time.dat};
          \addlegendentry{FISTA w/ BB};
      \addplot[very thick,red!90!black,solid,mark=star,mark repeat={100},mark size={\markSize},mark options={solid}] 
          table {figs/LASSO_test4_SPG_SpaRSA_time.dat};
          \addlegendentry{SPG/SpaRSA};
      \addplot[very thick,blue!80,solid,mark=diamond,mark repeat={100},mark size={\markSize},mark options={solid}] 
          table {figs/LASSO_test4_ZeroSR1_time.dat};
          \addlegendentry{0-mem SR1};
      \addplot[very thick,yellow!80!green,solid,mark=square,mark repeat={100},mark size={\markSize},mark options={solid}] %
          table {figs/LASSO_test4_ASA_time.dat};
          \addlegendentry{ASA};
      \addplot[very thick,green!80!black,solid,mark=oplus,mark repeat={100},mark size={\markSize},mark options={solid}] 
          table {figs/LASSO_test4_CGIST_time.dat};
          \addlegendentry{CGIST};
      \addplot[very thick,violet!95!black,solid,mark=otimes,mark repeat={100},mark size={\markSize},mark options={solid}] 
          table {figs/LASSO_test4_FPC_AS_time.dat};
          \addlegendentry{FPC-AS};
      \addplot[very thick,orange!97!black,solid,mark=+,mark repeat={100},mark size={\markSize},mark options={solid}] 
          table {figs/LASSO_test4_LBFGSB_time.dat};
          \addlegendentry{L-BFGS-B};
      \addplot[very thick,red!65!black,solid,mark=pentagon,mark repeat={100},mark size={\markSize},mark options={solid}] 
          table {figs/LASSO_test4_OWL_time.dat};
          \addlegendentry{OWL};
      \addplot[very thick,yellow!50!violet,solid,mark=x,mark repeat={100},mark size={\markSize},mark options={solid}] 
          table {figs/LASSO_test4_PSSas_time.dat};
          \addlegendentry{PSSas};
      \end{axis}
      \begin{scope}[xshift=0.47\linewidth]
      \begin{axis}[%
        width=0.44\linewidth,%
        height=6cm,%
        xmin=0,xmax=2.5,%
        ymode=log,
        xlabel={time in seconds},%
        legend columns=5,
        legend cell align=left,
        legend style={at={(1.2,1.02)},anchor=south,font=\footnotesize,column sep=5pt}
        ]
      \addplot[very thick,cyan!80!white,solid,mark=triangle,mark repeat={10},mark size={\markSize},mark options={solid}] %
          table {figs/LASSO_test5_FISTAwithBB_time.dat};
      \addplot[very thick,red!90!black,solid,mark=star,mark repeat={10},mark size={\markSize},mark options={solid}] 
          table {figs/LASSO_test5_SPG_SpaRSA_time.dat};
      \addplot[very thick,blue!80,solid,mark=diamond,mark repeat={10},mark size={\markSize},mark options={solid}] 
          table {figs/LASSO_test5_ZeroSR1_time.dat};
      \addplot[very thick,yellow!80!green,solid,mark=square,mark repeat={10},mark size={\markSize},mark options={solid}] %
          table {figs/LASSO_test5_ASA_time.dat};
      \addplot[very thick,green!80!black,solid,mark=oplus,mark repeat={30},mark size={\markSize},mark options={solid}] 
          table {figs/LASSO_test5_CGIST_time.dat};
      \addplot[very thick,violet!95!black,solid,mark=otimes,mark repeat={20},mark size={\markSize},mark options={solid}] 
          table {figs/LASSO_test5_FPC_AS_time.dat};
      \addplot[very thick,orange!97!black,solid,mark=+,mark repeat={20},mark size={\markSize},mark options={}] 
          table {figs/LASSO_test5_LBFGSB_time.dat};
      \addplot[very thick,red!65!black,solid,mark=pentagon,mark repeat={10},mark size={\markSize},mark options={solid}] 
          table {figs/LASSO_test5_OWL_time.dat};
      \addplot[very thick,yellow!50!violet,solid,mark=x,mark repeat={10},mark size={\markSize},mark options={solid}] 
          table {figs/LASSO_test5_PSSas_time.dat};
      \end{axis}
      \end{scope}
    \end{tikzpicture} 
    \else
      Use \texttt{ifcompilePGFfigstrue} to compile this figure!
    \fi
  \end{center}
  \caption{\label{fig:lasso}Convergence plots for the methods described in Section~\ref{sec:results} for solving the $\ell_1$ LASSO problem. The plot on the left corresponds to the experiment with the random matrix and the right plot to the experiment with the differential operator. The vertical axis is the same for both plots. The proposed 0-mem SR1 method and PSSas efficiently solves both problems. While our method generalizes easily to the $\ell_1-\ell_2$ sparsity norm, PSSas is hard to generalize.}
\end{figure}

Consider the unconstrained LASSO problem \eqref{eq:LASSO}.  Many codes, such as \cite{DhillonQuasiNewton} and L-BFGS-B~\cite{LBFGSB}, handle only non-negativity or box-constraints. Using the standard change of variables by introducing the positive and negative parts of $x$, the LASSO can be recast as 
\begin{equation} \label{eq:changeOfVariables}
  \min_{ x_{+}, x_{-} \ge 0 } \frac{1}{2}\| Ax_+ - Ax_- -b\|^2 + \lambda \ones^\top(x_+ + x_-)
\end{equation}
and then $x$ is recovered via $x=x_+ - x_-$. With such a formulation solvers such as L-BFGS-B are applicable.  However, this constrained problem has twice the number of variables, and the Hessian of the quadratic part changes from $A^\top A$ to $\tilde{A}=\begin{pmatrix}A^\top A & -A^\top A \\-A^\top A & A^\top A \end{pmatrix}$ which necessarily has (at least) $n$ degenerate 0 eigenvalues and adversely affects solvers.

A similar situation occurs with the hinge-loss function. Consider the shifted and reversed hinge loss function $ h(x) = \max( 0, x )$. Then one can split $x=x_+ - x_-$, add constraints $x_+ \ge0, x_- \ge 0$, and replace $h(x)$ with $ \ones^\top(x_+)$. As before, the Hessian gains $n$ degenerate eigenvalues.

We compared our proposed algorithm on the LASSO problem. The first example, on the left of Figure~\ref{fig:lasso}, is a typical example from compressed sensing that takes $A\in \RR^{m \times n}$ to have iid $\mathcal{N}(0,1)$ entries with $m=1500$ and $n=3000$. We set $\lambda=0.1$. L-BFGS-B does very well, followed closely by our proposed SR1 algorithm, PSSas, and FISTA. Note that L-BFGS-B and ASA are in Fortran and C, respectively (the other algorithms are in Matlab).

\begin{table} 
\scriptsize 
\begin{tabular}{l@{}l@{}ll}
\toprule
Acronym & Algorithm Name & Tests & Comments \\
\midrule
FISTA & Fast IST Algorithm & \S\ref{sec:num-exp-lasso},\ref{sec:num-exp-l1l2-lasso} & our own implementation in Matlab \\ 
SPG/SpaRSA\ \  & Spectral Projected Gradient\cite{SPG} as used in \cite{WrightSparsa08} &  \S\ref{sec:num-exp-lasso},\ref{sec:num-exp-l1l2-lasso} & Matlab version from \cite{WrightSparsa08} \\
\midrule
L-BFGS-B & Limited memory, box-constrained BFGS\cite{LBFGSB,L-BFGS-B-97}& \S\ref{sec:num-exp-lasso} & Fortran with Matlab wrapper \\
ASA & ``Active Set Algorithm'' (conjugate gradient) \cite{HagerZhang06} \ \ & \S\ref{sec:num-exp-lasso} & C with Matlab wrapper, ver.\ 2.2 \\
OWL & Orthant-wise Learning~\cite{AndrewGao07} & \S\ref{sec:num-exp-lasso} &  Active set; Matlab \\
PSSas & Projected Scaled Sub-gradient + Active Set~\cite{Schmidt2007a} & \S\ref{sec:num-exp-lasso} & Matlab \\
CGIST & ``CG + IST'' ~\cite{GoldsteinSetzer11} & \S\ref{sec:num-exp-lasso} & Matlab \\
FPC-AS & ``Fixed-point continuation + Active Set'' \cite{FPCAS} & \S\ref{sec:num-exp-lasso} & Matlab, ver.\ 1.21 \\
\midrule 
\textbf{0-mem SR1} & Algorithm \ref{alg:SR1} & \S\ref{sec:num-exp-lasso},\ref{sec:num-exp-l1l2-lasso} & \textbf{our approach} (in Matlab) \\
\bottomrule
\end{tabular} 
\caption{Algorithms used in experiments of sections \ref{sec:num-exp-lasso} and \ref{sec:num-exp-l1l2-lasso}.
The first two algorithms are standard ``first-order'' algorithms; the next group of algorithms use active-set strategies; and the final group of three algorithms use a diagonal $\pm$ rank-1 proximal mapping.
    Our implementation of FISTA used the Barzilai-Borwein stepsize~\cite{BB88} and line search, and restarted the momentum term every $1000$ iterations~\cite{DC15}. L-BFGS-B and ASA use the reformulation of \eqref{eq:changeOfVariables}.
For L-BFSG-B, we use the updated version~\cite{LBFGSB2011}.
Code for PSSas and OWL (slight variant of \cite{AndrewGao07}) from \cite{SchmidtThesis2010}.
}
\label{table:otherAlgorithms}
\end{table}
%
%

Our second example uses a square operator $A$ with dimensions $n = 15^3 = 3375$ chosen as a 3D discrete differential operator. This example stems from a numerical analysis problem to solve a discretized PDE as suggested by \cite{FletcherBB}. For this example, we set $\lambda=1$.  For all the solvers, we use the same parameters as in the previous example.  Unlike the previous example, the right of Figure~\ref{fig:lasso} now shows that L-BFGS-B is very slow on this problem. The FPC-AS method, very slow on the earlier test, is now the fastest. However, just as before, our SR1 method is nearly as good as the best algorithm. FISTA is significantly outperformed by our method on this problem. This robustness is one benefit of our approach, since the method does not rely on active-set identifying parameters and inner iteration tolerances. Moreover, the proposed SR1 method easily generalizes to other regularization terms.

\subsection{Group LASSO problem} \label{sec:num-exp-l1l2-lasso}

As a second experiment, we replace the $\ell_1$ sparsity term $\norm{x}_{1}$ in \eqref{eq:LASSO} with an $\ell_1-\ell_2$ sparsity $\norm{x}_{2,1}$ as in \eqref{eq:l1-l2-sparsity-term}, which is known to promote group sparsity (hence the name group LASSO). We partition the $N$ coordinates of $x\in\R^N$ into groups $b\in \mathscr B$ with randomly selected size $\abs{b}\leq 12$. For the numerical experiment, the entries of $A$ and $b$ are drawn uniformly in $[0,1]$, and we set $N=2500$, $M=1600$, and $\lambda=1$. As the $\ell_1-\ell_2$ norm is not polyhedral, active set based methods are hard to use. Also L-BFGS-B cannot be used, as the ``trick'' for the $\ell_1$-norm above does no apply here. The emerging rank-1 proximal mapping in our proposed proximal SR1 method can be solved efficiently as described in Section~\ref{sec:examples}. We apply Newton's method in the interval between breakpoints that locates the root. 

Figure~\ref{fig:Comparison-l2-l1-lasso} shows the convergence of several methods in terms of objective value error vs iteration (left plot) or time (right plot). \PO{Our 0SR1 method shows the best performance in the low and medium precision regime, while, for obtaining a high precision, accelerated strategies, such as FISTA, seem to be favorable. Presumably, this comes from the $\ell_2-\ell_1$ norm, which usually activates a whole block of coordinates, unlike in the LASSO case where eventually only a few coordinates are active and thus often has an improved condition number when restricted to these active variables. Acceleration strategies seem to compensate for this effect. In the beginning, the SR1 metric reflects the conditioning of the problem better than isotropic metrics.}

\PO{Figure~\ref{fig:Comparison-l2-l1-lasso} also suggests that the improvement with respect to FISTA could be further increased when a more efficient implementation of the diagonal $\pm$ rank-1 proximal mapping is used, or when the rank-1 update is combined with the acceleration strategy as in \cite{OP17}, which we will explore in future work.}

\begin{figure}[t]
  \begin{center}
  \ifcompilePGFfigs
    \begin{tikzpicture}
     \newcommand\markSize{3} 
     \newcommand\markRep{50} %
      \begin{axis}[%
        width=0.44\linewidth,%
        height=6cm,%
        xmin=0,xmax=1500,
        ymin=0.001, 
        ymode=log,
        xlabel={iteration},%
        ylabel={objective value error},%
        legend columns=3,
        legend cell align=left,
        legend style={at={(1.2,1.02)},anchor=south,font=\footnotesize,column sep=5pt}
        ]
      \addplot[very thick,cyan!80!white,solid,mark=triangle,mark repeat={\markRep},mark size={\markSize},mark options={solid}] %
          table {figs/GroupLasso_conv_FISTA_iter.dat};
          \addlegendentry{FISTA};
      \addplot[very thick,red!80!black,solid,mark=star,mark repeat={\markRep},mark size={\markSize},mark options={solid}] 
          table {figs/GroupLasso_conv_SPG-SpaRSA_iter.dat};
          \addlegendentry{SPG/SpaRSA};
      \addplot[very thick,blue!80,solid,mark=diamond,mark repeat={45},mark size={\markSize},mark options={solid}] 
          table {figs/GroupLasso_conv_zeroSR1c_LS_iter.dat};
          \addlegendentry{0-mem SR1};
      \end{axis}
      \begin{scope}[xshift=0.47\linewidth]
      \begin{axis}[%
        width=0.44\linewidth,%
        height=6cm,%
        xmin=0,xmax=50,
        ymin=0.001, 
        ymode=log,
        xlabel={time in seconds},%
        legend columns=3,
        legend cell align=left,
        legend style={at={(1.2,1.02)},anchor=south,font=\footnotesize,column sep=5pt}
        ]
     \newcommand\markRepRight{100} %
      \addplot[very thick,cyan!80!white,solid,mark=triangle,mark repeat={\markRepRight},mark size={\markSize},mark options={solid}] %
          table {figs/GroupLasso_conv_FISTA_time.dat};
      \addplot[very thick,red!80!black,solid,mark=star,mark repeat={\markRepRight},mark size={\markSize},mark options={solid}] 
          table {figs/GroupLasso_conv_SPG-SpaRSA_time.dat};
      \addplot[very thick,blue!80,solid,mark=diamond,mark repeat={\markRepRight},mark size={\markSize},mark options={solid}] 
          table {figs/GroupLasso_conv_zeroSR1c_LS_time.dat};
      \end{axis}      
      \end{scope}
    \end{tikzpicture} 
    \else
      Use \texttt{ifcompilePGFfigstrue} to compile this figure!
    \fi
  \end{center}
  \caption{\label{fig:Comparison-l2-l1-lasso}Convergence plots for the methods described in Section~\ref{sec:num-exp-l1l2-lasso} for solving the $\ell_1-\ell_2$ LASSO problem. The vertical axis is the same for both plots.  The methods based on the efficient solution of the diagonal $\pm$ rank-1 proximal mapping proposed in this paper outperform comparable methods based on a diagonally scaled proximal mapping.}
\end{figure}
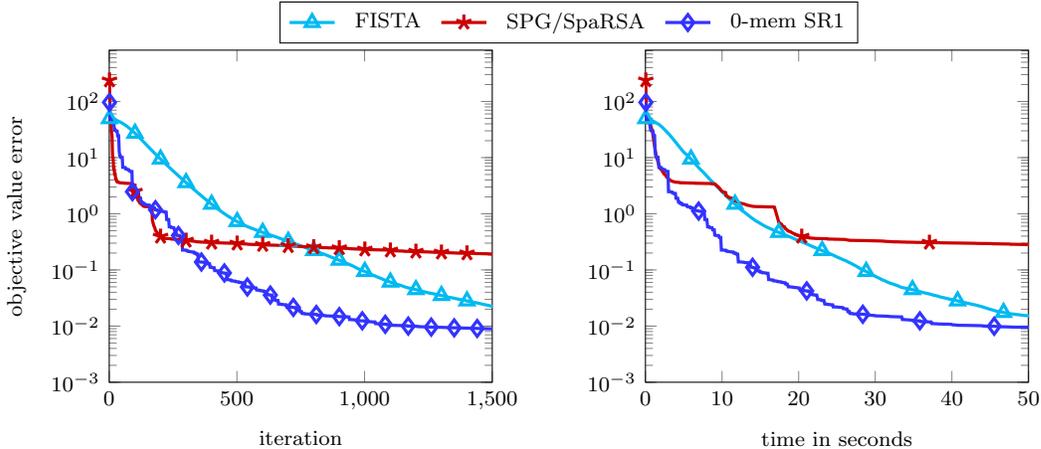

\section{Conclusions}
\label{sec:conclusion}

In this paper, we proposed a novel framework for variable metric (quasi-Newton) forward--backward splitting algorithms, designed to efficiently solve non-smooth convex problems structured as the sum of a smooth term and a non-smooth one. 
We introduced a class of weighted norms induced by diagonal $\pm$ rank $r$ symmetric positive definite matrices, as well as a calculus to compute the proximity operator in the corresponding induced metrics.
The latter result is new and generalized our previous results on the subject \cite{BF12}, and we believe it is of independent interest as even the simpler version from \cite{BF12} has been the basis of other works such as \cite{KV17,OP17}.
We also established convergence of the algorithm, and provided clear evidence that the non-diagonal term provides significant acceleration over diagonal matrices.

The proposed method can be extended in several ways. Although we focused on forward--backward splitting, our approach can be easily extended to the new {\em generalized} forward--backward algorithm of \cite{raguet-gfb}. However, if we switch to a primal-dual setting, which is desirable because it can handle more complicated objective functionals, updating $B_k$ is non-obvious, though one could perhaps use our results for a non-diagonal pre-conditioning method. 

Another improvement would be to derive efficient calculation for exact calculation of rank-2 proximity terms, thus allowing our 0-memory BFGS method to have cheaper and more exact update steps (as compared to the semi-smooth Newton method currently suggested). Theorem~\ref{theo:proxVrankr} and Corollary~\ref{cor:proxVrankr} give some clues in this direction. 

A final possible extension is to take $B_k$ to be diagonal plus rank-1 on diagonal blocks, since if $h$ is separable, this is still can be solved by our algorithm (see Proposition~\ref{prop:proxVseprank1}). The challenge here is adapting this to a robust quasi-Newton update. For some matrices that are well-approximated by low-rank blocks, such as H-matrices~\cite{Hmatrices}, it may be possible to choose $B_k \equiv B$ to be a fixed preconditioner.

\appendix
\pdfbookmark[0]{Appendix}{appendix} 

\section{Elements from convex analysis}
\label{sec:appendix}
We here collect some results from convex analysis that are key for our proof. Some lemmata are listed without proof and can be either easily proved or found in standard references such as ~\cite{Rockafellar70,BauschkeCombettes11}. 

\subsection{Background}
\paragraph*{Functions}

\begin{definition}[Indicator function]
\label{def:ind}
Let $\C$ a nonempty subset of $\Hm$. The indicator function $\indic_{\C}$ of $\C$ is 
\[
\label{eq:ind}
\indic_{\C} (x) =
  \begin{cases}
    0, & \text{if } x \in \C ~ ,\\
    +\infty, & \text{otherwise}.
  \end{cases}
\]
$\dom(\indic_\C)=\C$.
\end{definition}

\begin{definition}[Infimal convolution]  
Let $\f_1$ and $\f_2$ two functions from $\Hm$ to $\RR \cup \acc{+\infty}$. Their infimal convolution is the function from $\Hm$ to $\RR \cup \acc{\pm\infty}$ defined by:
\[
    (\f_1 \infc \f_2)(x) = \inf\acc{\f_1(x_1) + \f_2(x_2): x_1+x_2=x} = \inf_{y\in\Hm} \f_1 (y) + \f_2(x-y)~.
\]
\end{definition}

\paragraph*{Conjugacy}

\begin{definition}[Conjugate]
Let $\f: \Hm \to \RR \cup \acc{+\infty}$ having a minorizing affine function. The conjugate or Legendre-Fenchel transform of $\f$ on $\Hm$ is the function $\f^*$ defined by
\[
\label{eq:conj}
\f^*(v) = \sup_{x \in \dom(\f)} \pds{v}{x} - \f(x) ~ .
\]
\end{definition}

\begin{lemma}[Calculus rules]
\label{lem:conjcalc}
{~}\\
\vspace*{-0.5cm}
\begin{enumerate}[label={\rm (\roman{*})}, ref={\rm (\roman{*})}]
\item \label{conjaddcst} $(\f(x)+t)^*(v) = \f^*(v)-t$.
\item \label{conjscale} $(t \f(x))^*(v) = tf^*(v/t)$, $t > 0$.
\item \label{conjlin} $(\f \circ A)^* = \f^*\circ\parenth{A^{-1}}^*$ if $A$ is a linear invertible operator.
\item \label{conjtrans} $(\f(x-x_0))^*(v) = \f^*(v) + \pds{v}{x_0}$.
\item \label{conjsep} Separability: $\parenth{\sum_{i=1}^n \f_i(x_i)}^*(v_1,\cdots,v_n) = \sum_{i=1}^n \f_i^*(v_i)$, where $(x_1,\cdots,x_n)\in\Hm_1\times\cdots\times\Hm_n$.
\item \label{conjsum} Conjugate of a sum: assume $\f_1,\f_2 \in \Gamma_0(\Hm)$ and the relative interiors of their domains have a nonempty intersection. Then
\[
(\f_1 + \f_2)^* = \f_1^* \infc \f_2^*~.
\]
\item \label{conjHV} For $V \in \sdp(N)$, the conjugate of $f$ in $\Hm_V$ is $\f^*(V u)$.
\end{enumerate}
\end{lemma}

\begin{lemma}[Conjugate of a degenerate quadratic function]
\label{lem:conjquad}
Let $Q$ be a symmetric positive semi-definite matrix. Let $Q^+$ be its Moore-Penrose pseudo-inverse. Then,
\[
\label{eq:conjquad}
\parenth{\frac{1}{2}\norm{y - \cdot}^2_Q}^*(v) = 
\begin{cases}
\frac{1}{2} \norm{y - v}_{Q^{+}}^2 & \text{if } v \in y + \Span(Q) ~,\\
+\infty & \text{otherwise} ~.
\end{cases}
\]
\end{lemma}

\begin{lemma}[Conjugate of a rank-1 quadratic function]
\label{lem:conjrank1}
Let $u \in \Hm$. Then,
\[
\parenth{\frac{1}{2}\pds{u}{\cdot}^2}^*(v) = 
\begin{cases}
\frac{\norm{v}^2}{2\norm{u}^2} & \text{if } v \in \RR u ~,\\
+\infty & \text{otherwise} .
\end{cases}
\]
\end{lemma}

\paragraph*{Subdifferential}
\begin{definition}[Subdifferential]
\label{def:subdiff}
The subdifferential of a proper convex function $\f \in \Gamma_0(\Hm)$ at $x \in \Hm$ is the set-valued map $\partial \f: \Hm \to 2^{\Hm}$
\[
\label{eq:subdiff1}
\partial \f(x) = \left\{v \in \Hm | \forall z \in \Hm, \f(z) \geq \f(x) + \pds{v}{z-x}\right\} ~.
\]
An element $v$ of $\partial \f$ is called a subgradient. 
\end{definition}
The subdifferential map $\partial \f$ is a maximal monotone operator from $\Hm \to 2^{\Hm}$.

\begin{lemma}
\label{lem:subdiffgrad}
If $\f$ is (G\^ateaux) differentiable at $x$, its only subgradient at $x$ is its gradient $\nabla \f(x)$. 
\end{lemma}

\begin{lemma}
Let $V \in \sdp(N)$. Then $V \partial \f$ is the subdifferential of $\f$ in $\Hm_V$ .
\end{lemma}

The duality formulae to be stated shortly will be very useful throughout the rest of the paper.
\paragraph*{Fenchel duality}
\begin{lemma}
\label{lem:fencheldual}
Let $\f \in \Gamma_0(\Hm)$ and $g \in \Gamma_0(\Hm)$. Suppose that $0 \in \ri\parenth{\dom g - \dom \f}$. Then
\be
\label{eq:fencheldual1}
\inf_{x \in \Hm} \f(x) + g(x) = -\min_{u \in \Hm} \f^*(-u) + g^*(u) ~,
\ee
with the extremality relashionships between $\xs$ and $u^\star$, respectively the solutions of the primal and dual problems 
\be
\label{eq:fencheldual3}
\begin{split}
\xs 	 &\in \partial \f^*(-u^\star) 	\quad & \text{ and } & \quad u^\star \in \partial g(\xs) ~,\\
-u^\star &\in \partial \f(\xs) 		\quad & \text{ and } & \quad \xs \in \partial g^*(u^\star) ~.
\end{split}
\ee
\end{lemma}

\paragraph*{Toland duality}
\begin{lemma}
\label{lem:tolanddual}
Let $\f \in \Gamma_0(\Hm)$ and $g \in \Gamma_0(\Hm)$. Then
\be
\label{eq:tolanddual1}
\inf_{x \in \Hm} \f(x) - g(x) = \min_{u \in \Hm} g^*(u) - \f^*(u) ~.
\ee
If $\f - g$ is coercive, and $u^\star$ solves the dual problem in $u$, then there exists a solution $\xs$ of the primal problem and
\be
\begin{split}
\label{eq:tolanddual3}
\xs 	&\in \partial \f^*(u^\star)	\quad &\text{ and }& \quad u^\star \in \partial g(\xs) ~,\\
u^\star &\in \partial \f(\xs)		\quad &\text{ and }& \quad \xs 	 \in \partial g^*(u^\star) ~.
\end{split}
\ee
\end{lemma}
 
\begin{proof}
The first assertion is a consequence of \cite[Theorem~2.2]{Toland79}. The extremality relationships follow by combining \cite[Theorem~2.7 and 2.8]{Toland79}.

\end{proof}
\subsection{Proximal calculus in \texorpdfstring{$\Hm$}{H}}
\label{sec:appendix:prox}

\begin{definition}[Moreau envelope \cite{Moreau1962}]
\label{def:env} 
The function $\env{\f}{\rho}(x) = \inf_{z\in \Hm} \frac{1}{2\rho}\norm{x-z}^{2} + \f(z)$ for $0 < \rho < +\infty$ is the \textit{Moreau envelope} of index $\rho$ of $\f$.
\end{definition}

$\env{\f}{\rho}$ is also the infimal convolution of $\f$ with $\frac{1}{2\rho}\norm{\cdot}^2$.


\begin{lemma}
\label{lem:proxcalc}
{~}\\
\vspace{-0.5cm}
\begin{enumerate}[label={\rm (\roman{*})}, ref={\rm (\roman{*})}]
\item \label{proxtrans} Translation: $\prox_{\f(\cdot-y)}(x) = y + \prox_\f(x-y)$.
\item \label{proxscale} Scaling: $\forall \rho \in (-\infty,\infty), \prox_{\f(\rho \cdot)}(x) = \prox_{\rho^2f}(\rho x)/\rho$. 
\item \label{proxsep}   Separability~: let $(\f_i)_{1\le i\le n}$ a family of functions each in $\Gamma_0(\RR)$ and $\f(x) = \sum_{i=1}^N \f_i(x_i)$. Then $\f$ is in $\Gamma_0(\Hm)$ and $\prox_\f = \parenth{\prox_{\f_i}}_{1\le i \le N}$. 
\end{enumerate}
\end{lemma}

\begin{lemma}
\label{lem:envlip}
Let $\f \in \Gamma_0(\Hm)$. Then its Moreau envelope $\env{\f}{\rho}$ is convex and Fr\'echet-differentiable with $1/\rho$-Lipschitz gradient
\[
\nabla \env{\f}{\rho} = (\Id - \prox_{\rho \f})/\rho.
\]
\end{lemma}
\begin{lemma}[Moreau identity]
\label{lem:moreauident}
Let $\f \in \Gamma_0(\Hm)$, then for any $x \in \Hm$
\[
\prox_{\rho \f^*}(x) + \rho \prox_{\f/\rho}(x/\rho) = x, \forall ~ 0 < \rho < +\infty ~.
\]
\end{lemma}
From Lemma~\ref{lem:moreauident}, we conclude that
\[
\prox_{\f^*} = \Id - \prox_{\f}, \quad \prox_{\f^*}(x) \in \partial \f(x) ~.
\]


\section{Proofs of Section~\ref{sec:prox}}

\subsection{Proof of Lemma~\ref{lem:rel-prox-HV-to-H}} \label{appendix:lem:rel-prox-HV-to-H}
\begin{proof}
  Let $p=\prox^V_\f(x)$. The statement follows from the following equivalences
  \begin{eqnarray*}
      p = \prox_\f^V(x) 
      & \iff & x \in p + V^{-1}  \partial \f (p) \\
      & \iff & V^{1/2} x \in V^{1/2} p + V^{-1/2}\circ \partial \f \circ V^{-1/2} (V^{1/2} p) \\
      & \iff & V^{1/2} p = \prox_{\f\circ V^{-1/2}} (V^{1/2} x)\,.
  \end{eqnarray*}
\end{proof}

\subsection{Proof of Lemma \ref{lem:moreauidentV}} \label{appendix:lem:moreauidentV}
\begin{proof}
We have
\begin{eqnarray*}
p = \prox^{V}_{\rho \f^*}(x) = (\Id + V^{-1} \rho\partial \f^*)^{-1}(x) 
& \iff & V(x-p) \in \partial (\rho \f^*)(p)\\
& \iff & p \in \partial \f(V(x - p)/\rho)\\
& \iff & Vx/\rho - (Vx - Vp)/\rho \in V\partial (\f/\rho)(V(x - p)/\rho)\\
& \iff & V(x - p)/\rho = (\Id+V \partial (\f/\rho))^{-1}(Vx) \\
& \iff & x = p + \rho V^{-1} \circ (\Id+V \partial (\f/\rho))^{-1}(Vx) ~.
\end{eqnarray*}
\end{proof}

\subsection{Proof of Theorem~\ref{theo:proxVrankr}} \label{appendix:theo:proxVrankr}
\begin{proof}
Let $p=\prox^V_{\f}(x)$. Then, we have to solve
\begin{align}
  & \min_{z} \frac{1}{2} \norm{x - z}_V^2 + \f(z) \nonumber \\
  & \iff \min_{z} \parenth{\frac{1}{2}\norm{z}_\matP^2 - \pds{x}{z}_\matP + \f(z)} 
              \pm \frac{1}{2}\pds{x - z}{\matQ(x - z)} \nonumber \\
  \smat{y = \matP^{1/2}z \\ \matW=\matP^{-1/2}\matQ\matP^{-1/2}}  
  & \iff \min_{y} \parenth{\frac{1}{2}\norm{y}^2 - \pds{\matP^{1/2}x}{y} + \f\circ \matP^{-1/2}(y)} 
              \pm \frac{1}{2}\pds{\matP^{1/2}x - y}{\matW(\matP^{1/2}x - y)} \\
  \smat{\text{Lemma~\ref{lem:fencheldual}\eqref{eq:fencheldual1}}\\
        \text{or Lemma~\ref{lem:tolanddual}\eqref{eq:tolanddual1}}}  
  & \iff \min_{w} \pm \parenth{\frac{1}{2}\norm{\cdot}^2 - \pds{\matP^{1/2}x}{\cdot} + \f\circ \matP^{-1/2}}^*(\mp w) 
                + \parenth{\frac{1}{2}\pds{\matP^{1/2}x - \cdot}{\matW (\matP^{1/2}x - \cdot)}}^*(w) \nonumber\\
  \smat{\text{Lemma~\ref{lem:conjquad}}\\
        \text{and Lemma~\ref{lem:conjcalc}\ref{conjtrans}}} 
  & \iff \min_{w \in \Span(\matW)} \pm \parenth{\frac{1}{2}\norm{\cdot}^2 - \pds{\matP^{1/2}x}{\cdot} + \f\circ \matP^{-1/2}}^*(\mp w) 
          + \frac 12 \norm{w}^2_{\matW^+} + \pds{\matP^{1/2}x}{w} \nonumber\\
  \smat{\text{Lemma~\ref{lem:conjcalc}\ref{conjsum}-\ref{conjlin}}} 
  & \iff \min_{w \in \Span(\matW)} \pm \parenth{\parenth{\frac{1}{2}\norm{\cdot}^2 - \pds{\matP^{1/2}x}{\cdot}}^* \infc (\f^* \circ \matP^{1/2})}(\mp w) 
          + \frac 12 \norm{w}^2_{\matW^+} + \pds{\matP^{1/2}x}{w} \nonumber\\
  & \iff \min_{w \in \Span(\matW)} \pm \parenth{\parenth{\frac{1}{2}\norm{\matP^{1/2}x+\cdot}^2}  \infc (\f^* \circ \matP^{1/2})}(\mp w) 
          + \frac 12 \norm{w}^2_{\matW^+} + \pds{\matP^{1/2}x}{w} \nonumber\\
  \smat{\text{Definition~\ref{def:env}}}
  & \iff \min_{w \in \Span(\matW)} \pm \env{{\parenth{\f^*\circ \matP^{1/2}}}}{1}(\matP^{1/2}x \mp w) 
          + \frac 12 \norm{w}^2_{\matW^+} + \pds{\matP^{1/2}x}{w} ~.
\label{eq:mindual-rk-r}
\end{align}
By virtue of Lemma~\ref{lem:envlip}, $\env{{\parenth{\f^*\circ \matP^{1/2}}}}{{1}}$ is continuously differentiable with 1-Lipschitz gradient. Together with Lemma~\ref{lem:subdiffgrad}, Lemma~\ref{lem:fencheldual}\eqref{eq:fencheldual3} or Lemma~\ref{lem:tolanddual}\eqref{eq:tolanddual3}\footnote{The coercivity assumption holds (in fact the primal has exactly one solution) and the dual problem has indeed a non-empty set of minimizers.}, and Lemma~\ref{lem:moreauident}, this yields
\begin{eqnarray*}
p = \matP^{-1/2} \circ \nabla \env{{\parenth{\f^*\circ \matP^{1/2}}}}{{1}}(\matP^{1/2}x \mp w^\star)
&=& \matP^{-1/2} \circ \parenth{\Id - \prox_{\f^*\circ \matP^{1/2}}}(\matP^{1/2}x \mp w^\star)\\
&=& \matP^{-1/2} \circ \prox_{\f\circ \matP^{-1/2}}\circ \matP^{1/2}(x \mp \matP^{-1/2}w^\star),
\end{eqnarray*}
where $w^\star$ is a solution to the dual problem \eqref{eq:mindual-rk-r}, which will turn out to be unique as we will show shortly. Problem \eqref{eq:mindual-rk-r} is a minimization problem of a proper continuously differentiable objective with a Lipschitz continuous gradient over a linear set. The linear set can be parametrized by $\alpha\in\R^r$ such that $w=P^{-1/2}U \alpha$, and minimizing \eqref{eq:mindual-rk-r} is then equivalent to solving the $r$-dimensional smooth optimization problem
\begin{eqnarray}
\label{eq:minalpha-rank-r}
\min_{\alpha \in \R^r} \pm \env{{\parenth{\f^*\circ \matP^{1/2}}}}{{1}}(\matP^{1/2}x \mp \matP^{-1/2} U \alpha) + \frac 12 \norm{\alpha}^2_{U^\top\matQ^+U} + \pds{U^\top x}{\alpha} ~.
\end{eqnarray}
\JF{Since the columns of $U$ are linearly independent, $U^\top\matQ^+U$ is nothing but the identity operator on $\RR^r$.} The gradient of the objective in \eqref{eq:minalpha-rank-r} is given by the mapping $\lfun$. Lipschitz continuity of $\lfun$ follows from non-expansiveness of the proximal mapping, and the Lipschitz constant is straightforward from the triangle and Cauchy--Schwartz inequality. The root $\alpha^\star$ of $\lfun$ is unique if $\lfun$ is strongly monotone. In the case $V=P+Q$, strong monotonicity is immediate since all terms in \eqref{eq:minalpha-rank-r} are convex, and $\norm{\alpha}^2_{U^\top\matQ^+U}$ is strongly convex \JF{of modulus 1}.

In case $V=P-Q$, we apply Moreau's identity ($-\env{(\varphi^*)}{1}(x)=\env{\varphi}{1}(x) -\frac 12\norm{x}^2$ for $\varphi\in\Gamma_0(\Hm)$) (see, for example, \cite[Lemma 2.10]{CW05}) to the first term, which reduces the analysis of strong convexity to that of $\pds{\alpha}{(U^\top (Q^{+} - P^{-1})U)\alpha}$, hence, the positive definiteness of $U^\top (Q^{+} - P^{-1})U$. Since $P-Q\in\sdp(N)$, we have $\opnorm{P^{-1/2}QP^{-1/2}}<1$ and $P^{-1/2}QP^{-1/2}$ is invertible on $\Span(Q)$. Therefore, using $1= \opnorm{AA^{-1}}\leq\opnorm{A}\opnorm{A^{-1}}$ for an invertible matrix $A$, we conclude that $\opnorm{P^{1/2}Q^+P^{1/2}}_{\Span(Q)}>1$, where $\opnorm{\cdot}_{\Span(Q)}$ denotes the operator norm restricted to $\Span(Q)$, which implies that $Q^{+} - P^{-1} \in \sdp(\Span(Q))$ and, thus, \eqref{eq:minalpha-rank-r} is strongly convex. \JF{Its modulus of strong convexity is $\opnorm{U^\top (Q^{+} - P^{-1})U} = 1 - \opnorm{U^\top P^{-1}U}=1-\opnorm{P^{-1/2}U}^2$.}
\end{proof}

\subsection{Proof of Proposition~\ref{prop:bound}}
\label{appendix:prop:bound}
\SB{
\begin{proof}
    We use the notation of Theorem~\ref{theo:proxVrankr} and its proof in Appendix~\ref{appendix:theo:proxVrankr}.
    Let $p=\prox^V_{\f}(x)$.
    By non-expansivity of the proximal mapping, $\norm{p} = \norm{\prox^V_{\f}(x)}
    \le \norm{x} + \norm{\prox^V_{\f}(0)}$.
    Since $p$ minimizes $\frac{1}{2} \norm{x - z}_V^2 + \f(z)$, and using the same change of variable $y = \matP^{1/2}z$ as in the proof, the optimal point $y^\star = \matP^{1/2}p$.
    
    Letting $g(y) = \pm \frac{1}{2}\pds{\matP^{1/2}x - y}{\matW(\matP^{1/2}x - y)}$ with $ \matW=\matP^{-1/2}\matQ\matP^{-1/2}$, either Lemma~\ref{lem:fencheldual} or Lemma \ref{lem:tolanddual}
    gives the optimal dual solution
    \begin{align}
    \mp w^\star &=\nabla g(y^\star) \notag\\
    &= W(y^\star-P^{1/2}x) \notag \\
    &= \matP^{-1/2}\matQ\matP^{-1/2}y^\star 
    - \matP^{-1/2}\matQ x \notag \\
    &= \matP^{-1/2}\matQ(p - x ) . \label{eq:wpx}
    \end{align}
    Finally, $w^\star=P^{-1/2}U \alpha^\star$, and observe $U=u$ since $r=1$, and so also $Q=uu^\top$. Then \begin{align*}
    |\alpha^\star| &= \anorm{P^{1/2}w^\star}/\norm{u} \\
    &= \norm{\matQ(p - x )}/\norm{u} \quad\text{via~\eqref{eq:wpx}} \\
    &\le \norm{u}\parenth{ 2\norm{x} + \norm{\prox^V_{\f}(0)}} .
    \end{align*}
\end{proof}
}

\subsection{Proof of Proposition~\ref{prop:generaldh-rank-r}} \label{appendix:prop:generaldh-rank-r}
\JF{
\begin{proof}
The key of the proof is the remarkable stability properties of definable functions. In particular, under the sum, composition by a linear operator, derivation, and canonical projection (see~\cite{vandenDriesMiller96,coste1999omin}). Since $\f$ is a tame function, so is $\f \circ P^{-1/2}$, as well as its Moreau envelope (by the projection stability), and the gradient of the latter. Combining this with Lemma~\ref{lem:envlip}, it follows that $\prox_{\f \circ P^{-1/2}}$ is a tame mapping. We then deduce from stability to the sum and composition by a linear operator that $\lfun$ is a tame mapping. Thus, $\lfun$ is tame Lipschitz continuous mapping (Theorem~\ref{theo:proxVrankr}), and it follows from \cite[Theorem~1]{BolteSN09} that $\lfun$ is semi-smooth.

Let us now show that $\partial^C \lfun(\alpha^\star)$ is non-singular. By definition of the Clarke Jacobian for a Lipschitz function and the Carath\'eodory theorem, for any $G \in \partial^C \lfun(\alpha^\star)$, we have a finite sequence $\rho_1,\cdots, \rho_{r^2+1} \geq 0$ living on the simplex, i.e., $\sum_{i=1}^{r^2+1} \rho_i=1$, and $r^2+1$ sequences $\parenth{\alpha_{i,k}}_{k \in \NN}$ with $\alpha_{i,k} \underset{\Omega}{\to} \alpha^\star$ as $k \to +\infty$ such that, for any $d \in \RR^r$
\[
\pds{Gd}{d} = \sum_{i=1}^{r^2+1} \rho_i \lim_{k \to +\infty} \pds{\jacl(\alpha_{i,k})d}{d} = \sum_{i=1}^{r^2+1} \rho_i \lim_{k \to +\infty} \lim_{\tau \to 0} \frac{\pds{\lfun(\alpha_{i,k}+\tau d) - \lfun(\alpha_{i,k})}{d}}{\tau} .
\]
By strong monotonicity of $\lfun$ of modulus $c > 0$ (Theorem~\ref{theo:proxVrankr}), we have for all $d \in \RR^r$
\[
\frac{\pds{\lfun(\alpha_{i,k}+\tau d) - \lfun(\alpha_{i,k})}{d}}{\tau} \geq c \norm{d}^2 .
\]
Passing to the limit and summing, we conclude that
\[
\pds{Gd}{d} \geq c \norm{d}^2, \quad \forall d \in \RR^r .
\]
Since $G$ is any element of $\partial^C \lfun(\alpha^\star)$, we get that $\partial^C \lfun(\alpha^\star)$ is non-singular. We are then in position to apply \cite[Theorem~7.5.5]{FacchineiPang03} to obtain the first part of the convergence claim.

For the case where $\f$ is semi-algebraic, we argue as above, using stability of semi-algebraic sets to the same operations (in particular projection stability by the Tarski-Seidenberg principle~\cite{coste2002intro}), to deduce that $\prox_{\f \circ P^{-1/2}}$ is also a semi-algebraic mapping. The last claim then follows from~\cite[Theorem~2]{BolteSN09}.
\end{proof}
}

\subsection{Proof of Proposition~\ref{prop:proxVseprank1}} \label{appendix:prop:proxVseprank1}
\begin{proof}
  Recall that \eqref{eq:roothsep} is strictly increasing, continuous, and has a unique solution. When $\prox_{\f_i/d_i}$ is piecewise affine with $k_i$ segments, it is easy to see that $\lfun(\alpha)$ in \eqref{eq:roothsep} is also piecewise affine with slopes and intercepts changing at the $k^\prime$ (unique) transition points $\widetilde{\bm\theta}$. Therefore, the root of $\lfun$ can be found by sorting $\widetilde{\bm\theta}$ (Step~\ref{alg:root-pl-sep-sort-bpts}) and finding the interval between breakpoints that localizes the root (Step~\ref{alg:root-pl-sep-step-rootfinding}). Step~\ref{alg:root-pl-sep-sort-bpts} has the complexity $O(K \log(K))$. Step~\ref{alg:root-pl-sep-step-rootfinding} has the complexity $O(N\log(K))$, where $O(\log(K))$ steps are required for binary search and each step costs the evaluation of $\lfun$, which consists of $N$ terms. Step~\ref{alg:root-pl-sep-aff-accum} adds at most a complexity of $O(N)$. 
\end{proof}


\section{Proofs of Section~\ref{sec:SR1}}

\subsection{Proof of Lemma~\ref{lem:hkbnds}} \label{appendix:lem:hkbnds}
\begin{proof}
From \cite[p. 57 and 64]{NesterovBook}, we have for any $x$ and $y$ in $\dom(F)$
\begin{align*}
    L^{-1} \| \nabla f(x) - \nabla f(y) \|^2 &\le \< \nabla f(x) - \nabla f(y), x-y \> 
        \le L \|x-y\|^2 \\
        \mu \|x-y\|^2  &\le  \< \nabla f(x) - \nabla f(y), x-y \> 
        \le \mu^{-1}  \| \nabla f(x) - \nabla f(y) \|^2.
\end{align*}
Thus by applying the above results to the quasi-Newton sequences $s_k$ and $y_k$, we get 
\begin{equation}
\begin{aligned}
    L^{-1} \| y_k \|^2 &\le \< s_k, y_k \> \le L \|s_k\|^2  \\
    \mu \|s_k\|^2  &\le  \< s_k, y_k \> \le \mu^{-1}  \| y_k \|^2
\end{aligned}    
    \qquad \implies \qquad
\begin{aligned}
    L^{-1}  &\le \frac{\< s_k, y_k \>}{\|y_k\|^2} \le \mu^{-1} \\
    \mu   &\le  \frac{\< s_k, y_k \>}{\|s_k\|^2} \le L.
\end{aligned}
    \label{eq:Lipschitz}
\end{equation}
We will use the ``2nd'' Barzilai--Borwein stepsize $\BBa$ as opposed to the more common $\BBb$:
\begin{equation*}
    \BBa=\frac{\pds{s_k}{y_k}}{\|y_k\|^2}, \quad \BBb = \frac{\|s_k\|^2}{\pds{s_k}{y_k}}.
\end{equation*}
Via Cauchy-Schwarz, we have $\BBa \le \BBb$. From \eqref{eq:Lipschitz}, we have $L^{-1} \le \BBa \le \BBb \le \mu^{-1}$.  

Given the SR1 update and the choice $H_0 = \gamma \BBa \Id$ with $0 < \gamma < 1$, we have
\[
u_k = (s_k - \gamma \BBa y_k )/\sqrt{ \pds{s_k - \gamma \BBa y_k}{y_k} } =  (s_k - \gamma \BBa y_k )/\sqrt{ (1 - \gamma) \pds{s_k}{y_k} } .
\]
Combining this with the estimates \eqref{eq:Lipschitz}, we obtain
\begin{align*}
    \|u_k\|^2 &= \frac{ \|s_k\|^2 -2\gamma\BBa\pds{s_k}{y_k} + \gamma^2\BBa^2\|y_k\|^2 }{ (1-\gamma)\pds{s_k}{y_k} } \\
    &= (1-\gamma)^{-1}\left(  \frac{\|s_k\|^2}{\pds{s_k}{y_k}}
                            -2\gamma \BBa  + \gamma^2 \BBa \right) \\
    &\le (1-\gamma)^{-1}\left( \mu^{-1} - 2\gamma L^{-1} + \gamma^2 \mu^{-1} \right) ~.
\end{align*}
Thus
\begin{align*} 
0\prec \gamma L^{-1}\Id \preceq H_0 \preceq H_k &\preceq \gamma\mu^{-1}\Id + (1-\gamma)^{-1}\left( (1+\gamma^2)\mu^{-1} - 2\gamma L^{-1} \right)\Id \nonumber \\
&\preceq (1-\gamma)^{-1}\left( (1+\gamma) \mu^{-1} - 2\gamma L^{-1}  \right) \Id.
\end{align*}
\end{proof}

\subsection{Proof of Theorem~\ref{thm:convergence}} \label{appendix:thm:convergence}
\begin{proof}
We first recall the classical inequality for smooth functions with $L$-Lipschitz continuous gradient,
\begin{equation}
\label{eq:lipf}
f(x) - f(y) + \pds{\nabla f(y)}{y - x} \leq \frac{L}{2}\norm{x - y}^2 ~. 
\end{equation}
\paragraph*{$\bullet$ Case $\alpha \in ]0,1/2[$:}
It is clear that \eqref{eq:fbsr1} is equivalent to
\[
B_k(\xk - \xkk) - \stk \nabla f(\xk) \in \stk \partial h(\xk)
\]
which in turn implies
\begin{equation}
\label{eq:convineq}
h(y) \geq h(\xkk) + \stk^{-1} \pds{B_k(\xk - \xkk) - \stk \nabla f(\xk)}{y - \xkk}, \quad \forall y \in \dom(h) ~.  
\end{equation}
Applied at $\xk$, it yields
\begin{align}
\label{eq:dkl2}
h(\xk) - h(\xkk) + \pds{\nabla f(\xk)}{\xk - \xkk}	&\geq \stk^{-1} \norm{\xkk - \xk}_{B_k}^2 ~.
\end{align}
Denote $\dk = h(\xk) - h(\xkk) + \pds{\nabla f(\xk)}{\xk - \xkk}$. We have $\dk \geq 0$. In view of \eqref{eq:lipf}, we get
\begin{align*}
F(\xkk) - F(\xk) + \dk 	&= f(\xkk) - f(\xk) + \pds{\nabla f(\xk)}{\xk - \xkk} \\
						&\leq \frac{L}{2} \norm{\xkk - \xk}^2 \leq \frac{Lb}{2} \norm{\xkk - \xk}_{B_k}^2 ~,
\end{align*}
where we used Lemma~\ref{lem:hkbnds}. The last inequality together with \eqref{eq:dkl2} yields
\begin{align*}
F(\xkk) - F(\xk) &\leq -\parenth{1-\frac{Lb\stk}{2}} \dk \leq -\alpha \dk ~.
\end{align*}
By assumption, the right hand side is non-positive, meaning that the objective function decreases with $k$. Denote
\[
\ek = F(\xk) - F(\xs) \quad \text{and} \quad \delk = \ek - \ekk ~.
\]
Observe that $\ek$ is a positive and decreasing sequence, and thus converges. Moreover,
\[
\delk \geq \alpha \dk ~,
\]
Using convexity of $f$ and inequality \eqref{eq:convineq} at $y=\xs$, we obtain
\begin{align*}
\ek	&=   f(\xk) - f(\xs) + \pds{\nabla f(\xk)}{\xs - \xk} \\
	&~ + h(\xk) - h(\xkk) + \pds{\nabla f(\xk)}{\xk - \xkk} \\
	&~ + h(\xkk) - h(\xs) + \pds{\nabla f(\xk)}{\xkk - \xs} \\
	&\leq \dk + \stk^{-1} \pds{B_k(\xk - \xkk)}{\xkk - \xs} \\
	&\leq \dk + \stk^{-1} \norm{\xkk - \xs}_{B_k}\norm{\xkk - \xk}_{B_k} \\
	&\leq \dk + \sqrt{\frac{1}{\underline{\step} a}} \norm{\xkk - \xs}\sqrt{\dk} \\
	&\leq \alpha^{-1}\parenth{\delk + \sqrt{\frac{\alpha}{\underline{\step} a}} \norm{\xkk - \xs}\sqrt{\delk}} ~.
\end{align*}
Thus, using Young inequality, together with strong convexity of $f$ and $\ek$ is decreasing, we get for any $\veps > 0$,
\begin{align*}
\alpha\ek	&\leq \delk + \frac{\alpha\veps}{2\underline{\step} a} \norm{\xkk - \xs}^2 + \frac{\delk}{2\veps} \\
			&\leq  \parenth{1+\frac{1}{2\veps}}\delk + \frac{\alpha\veps}{\underline{\step} a\mu} \ekk \\
			&\leq  \parenth{1+\frac{1}{2\veps}}\delk + \frac{\alpha\veps}{\underline{\step} a\mu} \ek
			= \parenth{1+\frac{1}{2\veps}}\delk + \frac{\veps L \alpha}{\gamma\underline{\step}\mu} \ek ~.
\end{align*}
Let $\beta(\veps) = 1+\tfrac{1}{2\veps}$ and $\cnd=L/\mu > 1$. It follows that
\begin{align*}
\ekk 	&\leq \rho \ek ~, \qquad \rho = 1 - \frac{\alpha}{\beta(\veps)}\parenth{1 - \frac{\veps\cnd}{\gamma \underline{\step}}} ~.
\end{align*}
We always have $\beta(\veps) \in ]1,+\infty[$, and by assumption on the sequence $\stk$, $\alpha \in ]0,1[$. Choosing $\veps=\nu \gamma \underline{\step}/\cnd$, for any $\nu \in ]0,1[$, we get that $\rho = 1 - \alpha\frac{\nu(1-\nu)}{\nu + \eta} \in ]0,1[$. Therefore,
\[
\norm{\xk - \xs} \leq \sqrt{\frac{2\parenth{F(x_0) -(F\xs)}}{\mu}}\rho^{k/2} ~.
\]
The function $\nu \in ]0,1] \mapsto \nu(1-\nu)/(\nu + \eta)$ has a unique maximizer at $\nu_{\mathrm{opt}}=\sqrt{\eta^2+\eta}-\eta$ (which is indeed a strictly increasing function of $\eta$ on $]0,+\infty[$ taking values in $]0,1/2[$). We get the optimal rate $\rho_1$ by plugging $\nu_{\mathrm{opt}}$ into the expression of $\rho$.

\paragraph*{$\bullet$ Case $\alpha \in [1/2,1[$:}
From \eqref{eq:Q}, \eqref{eq:lipf}, Lemma~\ref{lem:hkbnds} and the assumption on $\alpha$, we have
\begin{align*}
Q_k^{B_k}(x) + h(x) 	&= F(x) - \parenth{f(x) - f(\xk) + \pds{\nabla f(\xk)}{\xk-x}} + \frac{1}{2\stk}\norm{x-\xk}_{B_k}^2 \\
					&\geq F(x) - \frac{L}{2}\norm{x - \xk}^2 + \frac{1}{2\stk}\norm{x-\xk}_{B_k}^2 \\
					&\geq F(x) + \frac{1}{2\stk}\parenth{1 - Lb\stk}\norm{x-\xk}_{B_k}^2 \\
					&\geq F(x) ~.
\end{align*}
Moreover, convexity of $f$ yields
\begin{align*}
Q_k^{B_k}(\xkk) + h(\xkk) 	
				&= \min_{x} Q_k^{B_k}(x) + h(x) \\
				&= \min_{x} F(x) - \parenth{f(x) - f(\xk) - \pds{\nabla f(\xk)}{x-\xk}} + \frac{1}{2\stk}\norm{x-\xk}_{B_k}^2 \\
				&\leq \min_{x} F(x) + \frac{1}{2\stk}\norm{x-\xk}_{B_k}^2 ~.
\end{align*}
It then follows that
\begin{align*}
F(\xkk)	&\leq  Q_k^{B_k}(\xkk) + h(\xkk) \\
		&\leq \min_{x} F(x) + \frac{1}{2\stk}\norm{x-\xk}_{B_k}^2 \\
		&\leq \min_{t \in [0,1]} F(t \xs + (1-t) \xk) + \frac{t^2}{2\stk}\norm{\xk-\xs}_{B_k}^2 \\
		&\leq \min_{t \in [0,1]} tF(\xs) + (1-t) F(\xk) + \frac{t^2L}{2\underline{\step}\gamma}\norm{\xk-\xs}^2 \\
		&\leq \min_{t \in [0,1]} F(\xk) - t\parenth{F(\xk) - F(\xs)} + \frac{t^2L}{\underline{\step}\gamma\mu}\parenth{F(\xk)-F(\xs)} \\
		&= \min_{t \in [0,1]} F(\xk) - t\parenth{1 - 2 t\eta}\parenth{F(\xk)-F(\xs)} ~.
\end{align*}
Thus, we arrive at
\begin{align*}
\ekk		&\leq \min_{t \in [0,1]} \parenth{1 - t\parenth{1 - 2 t\eta}}\ek = \rho_2 \ek ~.
\end{align*}
The function $\parenth{1 - t\parenth{1 - 2 t\eta}}$ attains its minimum uniquely at $t=1$ if $\eta \leq 1/4$, and $1/(4\eta)$ otherwise. Plugging these values gives the expression of $\rho_2$.
\end{proof}

\section{Proofs of Section~\ref{sec:BFGS}}

\subsection{Proof of Lemma~\ref{lem:hkbnds-bfgs}} \label{appendix:lem:hkbnds-bfgs}

\begin{proof}
  We derive a uniform bound for the matrix $H$ in \eqref{eq:BFGS-inv-Hess-formula-derivation}. Note that $u_\gamma$ from \eqref{eq:BFGS-inv-Hess-formula-derivation} satisfies with $\tau=\BBa$
  \[
    \begin{split}
    \rho \norm{u_\gamma}^2 
    =&\  \frac{\norm{s}^2 - 2\frac{\gamma\tau}{1+\gamma}\scal{y}{s} + \big(\frac{\gamma\tau}{1+\gamma}\big)^2 \norm{y}^2}{\scal{y}{s}} 
    =  \frac{\norm{s}^2}{\scal{y}{s}} - 2\frac{\gamma}{1+\gamma}\BBa + \frac{\gamma^2}{(1+\gamma)^2}  \BBa \\
    =&\  \BBb - \frac{\gamma}{1+\gamma}\Big(2 - \frac{\gamma}{1+\gamma}\Big) \BBa 
    \leq  \mu^{-1} - \frac{(2+\gamma)\gamma}{(1+\gamma)^2} L^{-1}
    \end{split}
  \]
  where we used $\rho^{-1} = \scal{y}{s}$, $\rho\norm{y}^2 = \BBa^{-1}$, and the estimations in the proof of Lemma~\ref{lem:hkbnds} for $\BBa$ and $\BBb$. Using this estimation, $\rho\BBa^2\norm{y}^2 = \BBa\geq L^{-1}$, and positive semi-definiteness of $yy^\top$ and $u_\gamma u_\gamma^\top$, we conclude that
  \[
    \begin{split}
    0
    \prec \frac{\gamma}{1+\gamma}L^{-1} \Id 
    = \gamma \Big(\BBa  - \frac{\gamma^2}{1+\gamma}\BBa\Big) \Id 
    \preceq H 
    \preceq&\ \BBa \gamma \Id + (1+\gamma) \mu^{-1} \Id - \frac{(2+\gamma)\gamma}{1+\gamma} L^{-1} \Id \\
    \preceq&\  (1+2\gamma) \mu^{-1} \Id  - \frac{(2+\gamma)\gamma}{1+\gamma} L^{-1} \Id .
    \end{split}
  \]
\end{proof}




\bibliographystyle{siam}

\small{
\bibliography{highordersplitting,auxiliaryRefs}
}

\end{document}